\documentclass[a4paper,12pt,reqno]{amsart}
\usepackage{amsmath,amsfonts,amssymb,amsthm,enumerate,multicol}
\usepackage{mathrsfs}
\usepackage{tikz,graphicx}
\usepackage{hyperref}
\usepackage{caption,enumitem,wasysym}
\usepackage{cleveref}
\usepackage{epstopdf}
\usepackage{float,wrapfig}
\usepackage[labelsep=space]{caption}
\usepackage[font=footnotesize]{caption}
\usepackage{xcolor}

\usepackage[english]{babel}
\usepackage[autostyle]{csquotes}

\usepackage{cite}

\theoremstyle{plain}
\newtheorem{thm}{Theorem}[section]

\newtheorem{lem}[thm]{Lemma}
\newtheorem{prop}[thm]{Proposition}

\newtheorem{defn}[thm]{Definition}

\usepackage{mathtools}

\usepackage{amsmath,amsfonts,amssymb,amsthm,enumerate,multicol}
\usepackage{tikz}
\usepackage{float}
\usepackage{autobreak}

\allowdisplaybreaks

\numberwithin{equation}{section}

\def\g{v}

\def\ff{u}
\def\kk{\Phi}
\def\bb{\beta}
\def\B{\beta_*}

\def\xx{{x_1}}
\def\yy{{x_2}}
\def\zz{{x_3}}

\def\mm{\boldsymbol\mu}  
\def\tf{\varsigma}

\begin{document}
	\begin{center}
		\Large{\textbf{The continuous collision-induced nonlinear fragmentation equation with non-integrable fragment daughter distributions }}
	\end{center}

	\medskip
	\medskip
	\centerline{${\text{Ankik Kumar Giri$^{\dagger*}$}}$, ${\text{Ram Gopal~ Jaiswal$^{\dagger}$}}$ and ${\text{Philippe Lauren\c{c}ot$^{\ddagger}$}}$}\let\thefootnote\relax\footnotetext{$^{*}$Corresponding author. Tel +91-1332-284818 (O);  Fax: +91-1332-273560  \newline{\it{${}$ \hspace{.3cm} Email address: }}ankik.giri@ma.iitr.ac.in}
	\medskip
	{\footnotesize
		
		\centerline{ ${}^{\dagger}\text{Department of Mathematics, Indian Institute of Technology Roorkee}$,}
		
		\centerline{Roorkee-247667, Uttarakhand, India}
		
		\centerline{ ${}^{\ddagger}\text{Laboratoire de Math\'ematiques (LAMA), UMR~$5127$}$,}
		\centerline{Universit\'e Savoie Mont Blanc, CNRS,}
		\centerline{ F--73000 Chamb\'ery, France}
	}

	\bigskip

	\begin{quote}
		{\small {\em \bf Abstract.}  Existence, non-existence, and uniqueness of mass-conserving weak solutions  to the continuous collision-induced nonlinear fragmentation equations are established for the  collision  kernels $\Phi$ satisfying $\Phi(x_1,x_2)={x_1}^{\lambda_1} {x_2}^{\lambda_2}+{x_2}^{\lambda_1} {x_1}^{\lambda_2}$, $(x_1,x_2)\in(0,\infty)^2$, with  ${\lambda_1} \leq {\lambda_2}\leq 1$, and non-integrable fragment daughter distributions. In particular, global existence of mass-conserving weak solutions is shown when $1\le\lambda:={\lambda_1}+{\lambda_2}\le2$ with $\lambda_1\ge k_0$,  the parameter $k_0\in(0,1)$ being related to the non-integrability of the fragment daughter distribution. The existence of at least one mass-conserving weak solution is also demonstrated when $2k_0 \le \lambda < 1$ with $\lambda_1\ge k_0$ but its maximal existence time is shown to be finite. Uniqueness is also established in both cases. The last result deals with the non-existence of mass-conserving weak solutions, even on a small time interval, for power law fragment daughter distribution when $\lambda_1<k_0$. It is worth mentioning that the previous literature on the nonlinear fragmentation equation does not treat non-integrable fragment daughter distribution functions.}
	\end{quote}
	\vspace{0.5cm}

	

	\section{Introduction}\label{sec:intro}
	
	The nonlinear fragmentation equation, sometimes known as collision-induced breakage equation, describes the dynamics of particles  which change their sizes in response to collisions with other particles and is given by the following nonlinear integro-differential equation
	\begin{equation}
		\partial_{t} \ff (t,\xx)=\mathcal{F} \ff (t,\xx)-\mathcal{L} \ff (t,\xx), ~~(t,\xx)\in (0,\infty)^2, \label{eq:main}
	\end{equation}
	\begin{equation}
		\ff (0,\xx)=\ff^{\mbox{\rm{\mbox{in}}}}(\xx)\geq0, ~~ \xx\in (0,\infty),  \label{eq:in}
	\end{equation}
	
	where the formation and loss term of particles with size $\xx$ are given by
	\begin{equation}
		\mathcal{F} \ff (\xx):=\frac{1}{2}\int_{\xx}^{\infty}\int_{0}^{\yy}\bb (\xx,\yy-\zz,\zz)\kk(\yy-\zz,\zz)\ff (\yy-\zz)\ff (\zz)d\zz d\yy \label{eq:formation}
	\end{equation}
	and
	\begin{equation}
		\mathcal{L} \ff (\xx):=\int_{0}^{\infty}\kk(\xx,\yy)\ff (\xx)\ff (\yy)d\yy, \label{eq:loss}
	\end{equation}
	for $\xx\in (0,\infty)$.
	
	In equations~\eqref{eq:main}-\eqref{eq:loss}, the quantity $\ff(t,\xx)$ represents the density of particles with size $\xx>0$ at time $t\geq 0$. The collision kernel $\kk(\xx,\yy)=\kk(\yy,\xx)\ge0$ describes the rate at which particles of sizes $\xx$ and $\yy$ collide, while the daughter distribution function $\bb(\zz,\xx,\yy)=\bb(\zz,\yy,\xx)\ge0$ gives the number of particles of size $\zz\in(0,\xx+\yy)$ that result from the collision between particles of sizes $\xx$ and $\yy$. In equation \eqref{eq:main}, the first term on the right-hand side accounts for the formation of particles with size $\xx$ due to the collision between particles of respective sizes $\yy-\zz$ and $\zz$, where $\yy>\xx$ and $\zz\in(0,\yy)$. The second term on the right-hand side represents the loss of particles of size $\xx$ due to their collision with particles of arbitrary size.
	
	Let us assume that local mass remains conserved, i.e.,
	\begin{equation}
		\int_{0}^{\xx+\yy}\zz\bb (\zz,\xx,\yy)d\zz=\xx+\yy ~~\text{and}~~ \bb (\zz,\xx,\yy)=0 ~\text{for}~ \zz>\xx+\yy, \label{eq:localc}
	\end{equation}
	for all $(\xx,\yy)\in (0,\infty)^2$. This is a fundamental property in the modeling of collisional breakup of particles. It ensures that no matter is created nor destroyed during the collision of particles. The property~\eqref{eq:localc} also leads to the global mass conservation principle, which is given by 
	\begin{equation}
		\int_{0}^{\infty}\xx\ff (t,\xx)d\xx=\int_{0}^{\infty}\xx\ff^{\mbox{\rm{\mbox{in}}}}(\xx)d\xx, ~t\geq 0. \label{eq:globalc}
	\end{equation}
	The identity~\eqref{eq:globalc} states that the total mass of particles at any time $t$ is equal to the initial total mass of particles.
	
	Furthermore, it can be inferred from equation~\eqref{eq:localc} that the combination of particles with sizes $\xx$ and $\yy$ through collision cannot produce fragments that are larger than the sum of their sizes $\xx + \yy$. However, it is reasonable to expect that collisional fragmentation between $\xx$ and $\yy$ may generate particles larger than $\xx$ and $\yy$ due to mass transfer from the smallest to the largest particles. However, if there exists a function $\B$ which is non-negative  and satisfies
	\begin{equation}
		\bb (\zz,\xx,\yy)=\B (\zz,\xx,\yy)\textbf{1}_{(0,\xx)}(\zz)+\B (\zz,\yy,\xx)\textbf{1}_{(0,\yy)}(\zz), \label{eq:masstransfer}
	\end{equation} 
	for $(\xx,\yy,\zz)\in (0,\infty)^3$ and 
	\begin{equation}
		\int_{0}^{\xx}\zz\B (\zz,\xx,\yy)d\zz=\xx, \label{eq:Blocalc}
	\end{equation}
	then there will be no mass transfer during the collision.
	
	The integrability of the non-negative function $\B$ is an important assumption, as it ensures that the number of daughter particles produced during the breakage of a particle of size $\xx$ is finite, which can be mathematically expressed as
	\begin{equation}
		\mathrm{N}(\xx,\yy):=\int_0^{\xx} \B (\zz,\xx,\yy)d\zz <\infty, ~ ~ (\xx,\yy)\in (0,\infty)^2. \label{eq:number_of_particles}
	\end{equation}
	 A specific type of breakage function, known as power law breakage, is given by
	\begin{equation}
		\B (\zz,\xx,\yy)=(\nu+2){\zz}^\nu {\xx}^{-\nu-1}\textbf{1}_{(0,\xx)}(\zz), \label{eq:powerlaw}
	\end{equation}
	 with $\nu\in (-2,0]$ \cite{ziff1987}, and satisfies the integrability assumption \eqref{eq:number_of_particles} when $\nu \in (-1,0]$, but fails to do so when $\nu \in (-2,-1]$.
	
	 To the best of our knowledge, the integrability assumption~\eqref{eq:number_of_particles} is a standing assumption in the studies of the nonlinear fragmentation equation~\eqref{eq:main} performed so far \cite{cheng1990, ernst2007, GL22021} and the main purpose of this work is to relax this assumption, thereby allowing to handle more singular fragment daughter distributions such as the one given by~\eqref{eq:powerlaw} with $\nu\in (-2,-1]$. We shall thus rather assume below that there are $k_0\in (0,1)$ and $p_{0}\in (1,1+k_0)$ such that $\B$ satisfies
	\begin{equation}
		\int_{0}^{\xx}{\zz}^{k_0}\B (\zz,\xx,\yy)^{p}d\zz\leq E_{k_0,p}{\xx}^{k_0+1-p}, \qquad p\in [1,p_0], \label{eq:nonintegrable}
	\end{equation}
	 for some positive constant $E_{k_0,p}$. It is worth pointing out that the assumption~\eqref{eq:nonintegrable} is a rather natural generalization of the integrability  {assumption}~\eqref{eq:number_of_particles} which somehow corresponds to the choice $k_0=0$ in~\eqref{eq:nonintegrable}. Furthermore, the power law distribution~\eqref{eq:powerlaw} satisfies~\eqref{eq:nonintegrable} for any $k_0 > |\nu| - 1$ and $p < \frac{k_0+1}{|\nu|}$.
	
	 As for the collision kernel~$\kk$, we assume for simplicity that it is given by
	\begin{equation}
		\kk(\xx,\yy)={\xx}^{\lambda_1} {\yy}^{\lambda_2}+{\yy}^{\lambda_1} {\xx}^{\lambda_2}, ~  k_0\leq {\lambda_1} \leq {\lambda_2}\leq 1,~ (\xx,\yy)\in(0,\infty)^2, \label{eq:kernel}
	\end{equation}
	for $k_0\in(0,1)$ stated in~\eqref{eq:nonintegrable} with homogeneity $\lambda={\lambda_1}+{\lambda_2}$.
	
	\medskip
	
	The nonlinear breakage/fragmentation equation has not been studied as completely as its linear counterpart. Over the past few decades, there has been a lot of interest in the linear fragmentation equation \cite{redner1990}, which was first investigated by Filippov \cite{Fil1961}, Kapur \cite{kapur1972}, McGrady and Ziff \cite{ziff1987, ziff1991}  and later studied by functional analytic methods in \cite{banasiak2002, banasiak2004, banasiak2006, bll2019, bt2018, eme2005} and by stochastic approaches in \cite{bert2002, haas2003}, see also the books \cite{bll2019, bert2006} for a more detailed account. 
	
	 In contrast, there are only a few studies available in the physics literature on the collision-induced breakage equation. An asymptotic analysis of the continuous collision-induced breakage events is performed by Cheng and Redner in \cite{cheng1988, cheng1990} for a class of models in which a collision of two particles causes both particles to split into two equal halves or either the largest particle, or the smallest one splits in two. In a later study, Krapivsky and Ben-Naim \cite{krapivsky2003} investigate the dynamics of collision-induced fragmentation, using its travelling wave behaviour to calculate the fragment mass distribution analytically. Kostoglou and Karabelas \cite{kostoglou2000, kostoglou2006} also study analytical solutions to the  collision-induced fragmentation and their asymptotic information using the Gamma distribution approximation. Ernst and Pagonabarraga \cite{ernst2007} describe the asymptotic behaviour of the collision-induced breakage equation for the product kernel $\kk(\xx,\yy) = (\xx \yy)^{\lambda/2}$ and $0\le \lambda \le 2$. In this case, the collision-induced breakage equation can be mapped to a linear breakage equation.
	
	Recently, in \cite{GL22021}, existence and uniqueness of mass-conserving weak solutions to~\eqref{eq:main}-\eqref{eq:in} are discussed for integrable daughter distribution functions and collision kernels of the form $\kk(\xx,\yy) = {\xx}^{\lambda_1} {\yy}^{\lambda_2} +{\xx}^{\lambda_2} {\yy}^{\lambda_1}$, where $\lambda := {\lambda_1}+{\lambda_2} \in [0,2]$. Moreover, the non-existence of weak solutions is also shown in \cite{GL22021} when $\lambda_1<0$.

	The goal of this article is to establish results on the existence, uniqueness and non-existence of mass-conserving weak solutions to \eqref{eq:main}-\eqref{eq:in} for a class of non-integrable daughter distribution functions. In particular, the existence of global weak solutions is shown for collision kernels that grow at least linearly while only local weak solutions are achieved for the case when collision kernels grow at most linearly. As in \cite{GL22021}, the existence proof relies on the weak $L^1$-compactness method introduced in \cite{stewart1989} for the coagulation-fragmentation equations and later developed in several papers, see \cite{bll2019} for a detailed account and further references. As for uniqueness, we shall follow the approach developed in \cite{eme2005, stewart1990} by showing a Lipschitz property in a suitably chosen weighted $L^1$-space. We finally adapt an argument from \cite{bll2019, CadC92, GL22021, vanD87c} to show the non-existence of mass-conserving weak solutions for some collision kernels.
	
	Before stating our results, we introduce the following notation:
	given a function $W$ which is non-negative and  measurable and defined on the interval  $(0, \infty),$ we set $X_W:= L^1((0, \infty), W(\xx)d\xx)$
	and
	\begin{equation*}
		 \lvert \lvert g \rvert \rvert_W: = \int_{0}^{\infty } \lvert g(\xx) \rvert W(\xx)d\xx, \quad \mm_{W}(g):=\int_{0}^{\infty }  g(\xx)  W(\xx)d\xx , ~g \in X_{W}.
	\end{equation*}
	
	The space $X_W$  endowed with its weak topology is denoted by $X_{W,w}$, while its positive cone is denoted by $X_{W,+}$. In particular, we denote  $X_k := X_{W_k}$  when $W(\xx) = W_k(\xx) := {\xx}^k, ~ \xx \in  (0, \infty),$ for some $k \in \mathbb{R}$ and
	\begin{equation*}
		\lvert \lvert g \rvert \rvert_{k} : = \int_{0}^{\infty } \xx^k \lvert g(\xx) \rvert d\xx, \quad	\mm_{k}(g):=\mm_{W_k}(g)=\int_{0}^{\infty }  {\xx}^k  g(\xx)d\xx, ~g \in X_{k}.
	\end{equation*}

\section{Main results}\label{sec:mr}
	
In order to state the main results of the paper, we first need to make clear what we mean by a weak solution to~\eqref{eq:main}-\eqref{eq:in}. The notion of weak solution used here differs slightly from that used for non-integrable fragment distributions in the classical coagulation-fragmentation equation \cite{bll2019, phl2018}. It requires the  integrability of function $({\xx}^{k_0}+{\yy}^{k_0})\kk(\xx,\yy)\ff (\tau,\xx)\ff (\tau,\yy)$ rather than integrability of function $\min\{\xx, \yy\}^{k_0}\kk(\xx,\yy)$ $\ff (\tau,\xx)\ff (\tau,\yy)$ for all $(\tau,\xx,\yy)\in (0,t)\times(0,\infty)^2$, where $t\in (0,T)$ and $T\in(0,\infty]$.

\subsection{Weak solution} 

\begin{defn}\label{defn:weaksolution}
		Let us fix $k_0\in(0,1)$, a non-negative and symmetric collision kernel $\Phi$ and a fragment daughter distribution $\beta$ satisfying~\eqref{eq:localc}, as well as
			\begin{equation*}
				\int_0^{\xx+\yy} \zz^{k_0} \beta(\zz,\xx,\yy) d\zz \le B \left( \xx^{k_0} + \yy^{k_0} \right), \qquad (\xx,\yy)\in (0,\infty)^2,
			\end{equation*}
		for some $B>0$, and consider an initial condition $\ff^{\mbox{\rm{\mbox{in}}}}\in X_{k_0,+}\cap X_{1}$ and $T\in(0,\infty].$ A weak solution to the nonlinear fragmentation equation \eqref{eq:main}-\eqref{eq:in} on $[0,T)$ is a non-negative function $\ff$ such that
		\begin{equation}
			\ff\in \mathcal{C}([0,T),X_{k_0,w})\cap L^{\infty}((0,T),X_{1}), \label{eq:ws1}
		\end{equation}
		with $\ff (0)=\ff^{\mbox{\rm{\mbox{in}}}}$ in $(0,\infty),$
		\begin{equation}
			(\tau,\xx,\yy) \longmapsto ({\xx}^{k_0}+{\yy}^{k_0})\kk(\xx,\yy)\ff (\tau,\xx)\ff (\tau,\yy)\in L^{1}((0,t)\times(0,\infty)^2),\label{eq:ws2}
		\end{equation}
		which also satisfies
		\begin{align}
			\int_{0}^{\infty}\tf(\xx)\ff (t,\xx)d\xx &=	\int_{0}^{\infty}\tf(\xx) \ff^{\mbox{\rm{\mbox{in}}}}(\xx)d\xx \nonumber \\
			&+\frac{1}{2}\int_{0}^{t}\int_{0}^{\infty}\int_{0}^{\infty}\Upsilon_{\tf}(\xx,\yy)\kk(\xx,\yy)\ff (\tau,\xx)\ff (\tau,\yy)d\yy d\xx d\tau,  \label{eq:wf}
		\end{align}
		for all $t\in (0,T)$ and $\tf \in \mathscr{T}^{k_0},$ where
		\begin{equation*}
			\mathscr{T}^{k_0}:=\{\tf\in \mathcal{C}^{0,k_0}([0,\infty))\cap L^{\infty}{((0,\infty)}): \tf(0)=0 \},
		\end{equation*}
		and 
		\begin{equation}
			\Upsilon_{\tf}(\xx,\yy)  = \Upsilon_{\tf}(\yy,\xx) :=\int_{0}^{\xx+\yy}\tf(\zz)\bb (\zz,\xx,\yy)d\zz-\tf(\xx)-\tf(\yy).  \label{eq:zeta}
		\end{equation}
		In addition, $u$ is said to be mass-conserving on $[0,T)$ if it satisfies
		\begin{equation*}
			\mm_1(\ff(t)) = \mm_1(\ff^{\mbox{\rm{\mbox{in}}}}), \qquad t\in [0,T).
		\end{equation*}
	\end{defn}
	
	First, we establish that there exists at least one weak solution to \eqref{eq:main}-\eqref{eq:in} on some time interval $[0,T_*)$ in the sense of \Cref{defn:weaksolution}, which is also mass-conserving. The result established here provides an extension of \cite{GL22021} to the class of daughter distributions which are non-integrable.

	\begin{thm}\label{maintheorem}
		
		Suppose that $\kk$ is defined according to equation~\eqref{eq:kernel} and that $\bb$ fulfills conditions \eqref{eq:masstransfer}, \eqref{eq:Blocalc}, and \eqref{eq:nonintegrable} for some fixed $k_0\in (0,1)$.
		Consider an initial condition $\ff^{\mbox{\rm{\mbox{in}}}} \in  X_{k_0,+} \cap X_{1}$ with a positive mass $\rho := \mm_1(\ff^{\mbox{\rm{\mbox{in}}}})$ such that
		\begin{equation}
			\int_0^\infty {\xx}^{k_0+1} \ff^{\mbox{\rm{\mbox{in}}}}(\xx) d\xx < \infty. \label{eq:additionalinitiadata}
		\end{equation}
		Then there exists at least one mass-conserving weak solution  $\ff$ to~\eqref{eq:main}-\eqref{eq:in} on $[0, T_*)$ in the sense of \Cref{defn:weaksolution}, where
		\begin{equation*}T_*=\begin{cases}
			T_{k_0} \qquad \text{if}~ \lambda \in (0,1),\\
			\infty  \qquad \text{if}~ \lambda \in [1,2],
		\end{cases}\end{equation*}
		 and $T_{k_0}$ is defined in \Cref{thm:smallsizebehavious}, see~\eqref{eq:Tk0}, and only depends on $\ff^{\mbox{\rm{\mbox{in}}}}$, $\rho$, and  $\lambda=\lambda_1+\lambda_2$.
	\end{thm}
	
	Theorem~\ref{maintheorem} is proved using a weak compactness approach in the space  $X_{k_0}\cap X_{1+k_0}$, a method pioneered in \cite{stewart1989} for the coagulation-fragmentation equation, and adapted in particular to~\eqref{eq:main} in \cite{GL22021} for integrable fragment daughter distributions. For $\ff^{\mbox{\rm{\mbox{in}}}} \in X_{0,+} \cap X_1$, it is known from \cite{GL22021} that both global and local mass-conserving weak solutions to \eqref{eq:main}-\eqref{eq:in} exist under the assumptions~\eqref{eq:masstransfer}, \eqref{eq:Blocalc}, \eqref{eq:nonintegrable}, and \eqref{eq:kernel} with  $k_0=0$. 

		In addition to the existence of mass-conserving weak solutions, a result on uniqueness is  demonstrated for a more restricted set of initial data $\ff^{\mbox{\rm{\mbox{in}}}}$.
		\begin{thm}\label{thm:uniqueness}
			Suppose that $\kk$ is defined according to equation~\eqref{eq:kernel} and that $\bb$ fulfills conditions~\eqref{eq:masstransfer}, \eqref{eq:Blocalc}, and~\eqref{eq:nonintegrable} for some fixed $k_0\in (0,1)$. Consider an initial condition $\ff^{\mbox{\rm{\mbox{in}}}} \in X_{k_0,+} \cap X_{k_0+1}$ and $T \in (0,\infty]$. Then there exists a unique weak solution $\ff$ to the equation~\eqref{eq:main}-\eqref{eq:in} on the time interval $[0,T)$ satisfying the condition  
			\begin{equation}
				\mm_{1+k_0+{\lambda_2}}(\ff) \in L^1(0,t) ~~\text{for every}~~ t \in (0,T). \label{eq:moment1+k0+lambda2}
			\end{equation}
		\end{thm}

		In order to prove uniqueness, similar to \cite{bll2019, eme2005, giri2013, GL22021, stewart1990}, it is necessary to control the distance between two solutions in a suitable weighted $L^1$-space. The challenging aspect of this proof is to select an appropriate weight function, which in this case is given by $w(\xx):= \max\{{\xx}^{k_0},{\xx}^{1+k_0}\}$ for $\xx>0$.
		
		\bigskip
		
		 We finally identify a range of the parameters $(\lambda_1,\lambda_2)$, along with a class of power law fragment daughter distributions~\eqref{eq:powerlaw}, for which no non-zero mass-conserving weak solutions to equation~\eqref{eq:main}-\eqref{eq:in} on $[0,T)$ exists whatever the value of $T>0$.
		
		\begin{thm} \label{thm:nonexistence} Suppose that $\kk$ is defined according to equation~\eqref{eq:kernel} and that $\bb$ fulfills conditions~\eqref{eq:masstransfer}, \eqref{eq:Blocalc}, and~\eqref{eq:powerlaw} for some $\nu\in (-2,-1]$. Fix $k_0\in (|\nu|-1,1)$ and consider an initial condition $\ff^{\mbox{\rm{\mbox{in}}}} \in  X_{k_0,+} \cap  X_{1+k_0}$ with $\rho  := \mm_1(\ff^{\mbox{\rm{\mbox{in}}}}) > 0$ and $T>0$. If ${\lambda_1}<|\nu|-1$ and $\lambda=\lambda_1+\lambda_2<1$, then \eqref{eq:main}-\eqref{eq:in} does not admit a mass-conserving weak solution on $[0,T)$.
		\end{thm}
		
		 A similar result is established in \cite[Theorem~1.8]{GL22021} for integrable fragment daughter distributions~\eqref{eq:powerlaw} with $\nu\in (-1,0]$ when $\lambda_1<0$, extending an observation from \cite{ernst2007} when $\lambda_1=\lambda_2<0$. The novelty here is that the non-existence result is not restricted to collision kernels involving at least a negative exponent. In particular, \Cref{thm:nonexistence} applies to the constant collision kernel (corresponding to $\lambda_1=\lambda_2=0$) as soon as $\nu\in (-2,-1)$. As in \cite{GL22021}, the proof of \Cref{thm:nonexistence} is adapted from \cite{bll2019, CadC92, vanD87c}. 
		
		\bigskip
		
		Let us now describe the content of the paper. The following section contains additional properties of weak solutions to \eqref{eq:main}-\eqref{eq:in} based on \Cref{defn:weaksolution}, including a criterion for mass conservation provided by \Cref{prop:massconservation}, along with some tail control in \Cref{lem:mtail}. In \cref{sec:eca}, we provide the proof of \Cref{maintheorem} by a weak $L^{1}$ compactness approach. In order to prove \Cref{maintheorem}, first we prove that \eqref{eq:main} is well-posed when $\bb$ satisfies~\eqref{eq:masstransfer}, \eqref{eq:Blocalc} and~\eqref{eq:nonintegrable} for suitably truncated collision kernels with the help of Banach fixed point theorem. To avoid concentration at a finite size and prevent the escape of matter, a uniform integrability estimate is derived along with estimates for both small and large sizes. The Dunford-Pettis theorem ensures that the approximating sequence is weakly compact with respect to size. In the next step, we estimate time equicontinuity to obtain the compactness with respect to time. \Cref{sec:u} is devoted to the proof of \Cref{thm:uniqueness}. We end the paper with the proof of the non-existence result in \cref{sec:ne}.

		\section{Fundamental properties}\label{sec:fp}
		
		First, let us recall that \Cref{defn:weaksolution} makes sense due to the regularity properties of  $\beta$ described in \Cref{defn:weaksolution}.
		
		\begin{lem}\label{lem:m}
			Let ${k_0}\in(0,1)$ and  $\beta$ satisfies~\eqref{eq:masstransfer}, \eqref{eq:Blocalc}, and~\eqref{eq:nonintegrable}. Consider a function $\tf\in C^{0,k_0}([0,\infty))$ satisfying $\tf(0)=0$. Then 
			\begin{equation*}
				\lvert \Upsilon_{\tf}(\xx,\yy) \rvert\leq E_{k_0}({\xx}^{k_0}+{\yy}^{k_0})||\tf||_{C^{0,k_0}}, ~(\xx,\yy)\in(0,\infty)^2,
			\end{equation*}
			where $E_{k_0}:=E_{k_0,1}+1$. 
		\end{lem}
		
		\begin{proof}
			The proof directly follows from the definition of $\Upsilon_{\tf}$, \eqref{eq:masstransfer}, \eqref{eq:Blocalc}, and \eqref{eq:nonintegrable}.
		\end{proof}
		
		\subsection{Conservation of mass} In the following proposition, we demonstrate the conditions under which a weak solution to equation~\eqref{eq:main}-\eqref{eq:in} is mass-conserving.

		\begin{prop} \label{prop:massconservation}
			 Let $T\in (0,\infty]$ and $k_0\in (0,1)$. Suppose that $\ff$ is a weak solution to~\eqref{eq:main}-\eqref{eq:in} on $[0, T)$ with initial condition $\ff^{\mbox{\rm{\mbox{in}}}} \in X_{k_0,+} \cap X_1$  in the sense of \Cref{defn:weaksolution} satisfying additionally
			\begin{equation}
				\int_0^t \int_{0}^{\infty} \int_{0}^{\infty}(\xx+\yy)\kk(\xx,\yy)\ff (\tau,\xx)\ff (\tau,\yy)d\xx d\yy d\tau <\infty \label{eq:mass_conservation_condition}
			\end{equation}
			for all $t \in (0, T)$. Then $\ff$ is a mass-conserving  weak solution to \eqref{eq:main}-\eqref{eq:in} on $[0, T)$.
		\end{prop}
		
		\begin{proof}
			The proof is the same as that of \cite[Proposition~2.2]{GL22021}, to which we refer.
		\end{proof}

		\begin{lem}
		 Let $T\in (0,\infty]$ and $k_0\in (0,1)$, and consider a weak solution $\ff$ to~\eqref{eq:main}-\eqref{eq:in} on $[0,T)$ with initial condition $\ff^{\mbox{\rm{\mbox{in}}}} \in X_{k_0,+} \cap X_1$ in the sense of \Cref{defn:weaksolution} which also satisfies~\eqref{eq:mass_conservation_condition}. Then the validity of the weak formulation~\eqref{eq:wf} extends to any function 
		\begin{equation*}
		\tf \in \mathcal{C}^{0,k_0}([0,\infty))~\text{ with} ~  \tf(0)=0~\text{and}~ \sup_{\xx>1}\left\{ \frac{|\tf(\xx)|}{\xx} \right\} < \infty.
		\end{equation*}
		\end{lem}


		\begin{proof}
			Consider $\tf \in \mathcal{C}^{0,k_0}([0,\infty))$ such that $\tf(0)=0$ which satisfies $|\tf(\xx)| \le C\xx$ for $\xx > 1$, where $C$ is a positive constant. Observe that 
			\begin{equation}
				|\tf(\xx)| \le C(\xx^{k_0}+\xx)~\text{for}~ \xx > 0. \label{remark_eq1}
			\end{equation}
			
		For $L>0$, let  us define
		 \begin{equation*}
		 	\tf_L (\xx)=\tf(\xx) \textbf{1}_{(0,L)}(\xx)+\tf(L) \textbf{1}_{(L,\infty)}(\xx), \qquad \xx>0,
		 \end{equation*}
		  and note that $\tf_L$ also satisfies~\eqref{remark_eq1}. Also we can easily see that $\tf_L \in \mathscr{T}^{k_0} $, so that, from \eqref{eq:wf}, we can write for $L>0$,
		 
		 	\begin{align}
		 	\int_{0}^{\infty}\tf_L(\xx)\ff (t,\xx)d\xx & = \int_{0}^{\infty}\tf_L(\xx) \ff^{\mbox{\rm{\mbox{in}}}}(\xx)d\xx \nonumber \\
		 	&\hspace{-0.5cm}+\frac{1}{2}\int_{0}^{t}\int_{0}^{\infty}\int_{0}^{\infty}\Upsilon_{\tf_L}(\xx,\yy)\kk(\xx,\yy)\ff (\tau,\xx)\ff (\tau,\yy)d\yy d\xx d\tau. \label{remark_eq3}
		 \end{align}
		 Now, using~\eqref{eq:masstransfer}, \eqref{eq:Blocalc}, \eqref{eq:nonintegrable}, \eqref{eq:zeta}, and \eqref{remark_eq1} (for $\tf_L$), we get 
		 \begin{equation}
		 	|\Upsilon_{\tf_L}(\xx,\yy)|\le  {CE_{k_0}}(\xx^{k_0}+\yy^{k_0}+\xx+\yy), \label{remark_eq4}
		 \end{equation}
		  since~\eqref{eq:Blocalc} and~\eqref{eq:nonintegrable} imply that $E_{k_0,1}\ge 1$. Now, with the help of~\eqref{eq:ws2},  \eqref{eq:mass_conservation_condition}, and~\eqref{remark_eq4}, we may apply the Lebesgue dominated convergence theorem and pass to the limit as $L \to \infty$ in~\eqref{remark_eq3} to obtain		 
		 	\begin{align}
		 	\int_{0}^{\infty}\tf(\xx)\ff (t,\xx)d\xx &=	\int_{0}^{\infty}\tf(\xx) \ff^{\mbox{\rm{\mbox{in}}}}(\xx)d\xx \nonumber \\
		 	&+\frac{1}{2}\int_{0}^{t}\int_{0}^{\infty}\int_{0}^{\infty}\Upsilon_{\tf}(\xx,\yy)\kk(\xx,\yy)\ff (\tau,\xx)\ff (\tau,\yy)d\yy d\xx d\tau  \label{remark_eq5}
		 \end{align}
		  and complete the proof.
		 \end{proof}

			In particular, when $0 < {L_1} < {L_2}$, the function
		\begin{equation*}\tf(\xx)={\xx}^{k} \textbf{1}_{({L_1},L_2)}(\xx)+\xx {L_2}^{k-1} \textbf{1}_{({L_2},\infty)}(\xx)\end{equation*} 
		defined for $\xx\in(0,\infty)$ and $k\ge1$ can be employed as a test function in equation~\eqref{eq:wf}.
		
		\subsection{Control on  linear and superlinear moments for large sizes} Control on  the first moment, which is nothing but control on mass for large sizes, has been noted in \cite{ernst2007} and rigorously proven in  \cite[Lemma~2.3]{GL22021} by using the test function $\tf(\xx)=\xx \textbf{1}_{(0,{L_1})}(\xx)$ for ${L_1}>0$. Similarly, we establish a control on superlinear moments for large sizes.

		\begin{lem} \label{lem:mtail}
			 Let $T\in (0,\infty]$, $k_0\in (0,1)$, and a non-negative and symmetric collision kernel $\bar{\Phi}$ satisfying $\bar{\Phi}\le\Phi$. Consider a mass-conserving weak solution $\ff$ to the equation~\eqref{eq:main}-\eqref{eq:in} on the time interval $[0, T)$ with $\bar{\Phi}$ instead of $\Phi$. Let the initial condition $\ff^{\mbox{\rm{\mbox{in}}}}\in X_{k_0,+} \cap X_1$. Assuming that $\bb$ satisfies~\eqref{eq:masstransfer} and~\eqref{eq:Blocalc}, and that $\ff^{\mbox{\rm{\mbox{in}}}}\in X_{k}$ for some $k\geq 1$, we have
			\begin{align}
				\int_{\xx}^{\infty}{\yy}^{k}\ff (t,\yy)d\yy \leq \int_{\xx}^{\infty}{\yy}^{k}\ff^{\mbox{\rm{\mbox{in}}}}
				(\yy)d\yy,
			\end{align}
			for all $(t, \xx) \in [0, T) \times (0, \infty)$. In particular,
			\begin{align}
			\mm_k (\ff(t))\le \mm_k(\ff^{\mbox{\rm{\mbox{in}}}}),  	\qquad t \in [0, T).
			\end{align}
		\end{lem}

		\begin{proof}  Let $0<{L_1}<{L_2}$ and consider the function 
			\begin{equation*}
			\tf(\xx)={\xx}^{k} \textbf{1}_{({L_1},L_2)}(\xx)+\xx {L_2}^{k-1} \textbf{1}_{({L_2},\infty)}(\xx), \qquad \xx\in(0,\infty).
			\end{equation*}
			First we  study the properties of the function $\Upsilon_\tf$ defined in~\eqref{eq:zeta}, using \eqref{eq:masstransfer} and \eqref{eq:Blocalc}.

				\textit{Case~1.}  For $(\xx,\yy)\in (0,{L_1})^2$, we have $\tf(\xx)=0$ and $\tf(\yy)=0$, so that 
				\begin{align*}
					\Upsilon_\tf(\xx,\yy)=0.
				\end{align*}
				
			\textit{Case~2.} For $(\xx,\yy)\in (0,{L_1})\times ({L_1},{L_2}) $, we have $\tf(\xx)=0$ and $\tf(\yy)={\yy}^{k}$, so that 
				\begin{align*}
					\Upsilon_\tf(\xx,\yy)&=\int_{L_1}^{\yy} {\zz}^{k} \B (\zz,\yy,\xx)d\zz-{\yy}^{k}\\
					&\leq  {\yy}^{k-1}\int_0^{\yy} \zz \B (\zz,\yy,\xx)d\zz -{\yy}^{k}=0.
				\end{align*}
				
				\textit{Case~3.} For $(\xx,\yy)\in ({L_1},{L_2})\times (0,{L_1}) $, we have   $\Upsilon_\tf(\xx,\yy) = \Upsilon_\tf(\yy,\xx)$ and we deduce from \textit{Case~2} that $\Upsilon_\tf(\xx,\yy)\le 0$.
				
				\textit{Case~4.} For $(\xx,\yy)\in ({L_1},{L_2})^2 $, we have  $\tf(\xx)={\xx}^{k}$ and $\tf(\yy)={\yy}^{k}$, so that 
				\begin{align*}
					\Upsilon_\tf(\xx,\yy)&=\int_{L_1}^\xx {\zz}^{k} \B (\zz,\xx,\yy)d\zz+\int_{L_1}^\yy {\zz}^{k} \B (\zz,\yy,\xx)d\zz-{\xx}^{k}-{\yy}^{k}\\
					&\leq 0.
				\end{align*}

				\textit{Case~5.} For $(\xx,\yy)\in (0,{L_1})\times ({L_2},\infty) $, we have  $\tf(\xx)=0$ and $\tf(\yy)=\yy {L_2}^{k-1}$, so that 
				\begin{align*}
					\Upsilon_\tf(\xx,\yy)&=\int_{L_1}^{L_2} \zz ^{k}\B (\zz,\yy,\xx)d\zz+\int_{L_2}^\yy \zz {L_2}^{k-1} \B (\zz,\yy,\xx)d\zz-\yy{L_2}^{k-1}\\
					& \leq {L_2}^{k-1} \int_{L_1}^\yy \zz\B (\zz,\yy,\xx)d\zz-\yy{L_2}^{k-1} \leq 0.
				\end{align*}
				
				\textit{Case~6.} For $(\xx,\yy)\in ({L_2},\infty) \times  (0,{L_1})$, we have  $\Upsilon_\tf(\xx,\yy) = \Upsilon_\tf(\yy,\xx)$ and we deduce from \textit{Case~5} that $\Upsilon_\tf(\xx,\yy)\le 0$.

				\textit{Case~7.} For $(\xx,\yy)\in ({L_1},{L_2})\times ({L_2},\infty) $, we have  $\tf(\xx)={\xx}^{k}$ and $\tf(\yy)=\yy {L_2}^{k-1}$, so that 
				\begin{align*}
					\Upsilon_\tf(\xx,\yy)&=\int_{L_1}^\xx {\zz}^{k} \B (\zz,\xx,\yy)d\zz+\int_{L_1}^{L_2} \zz ^{k}\B (\zz,\yy,\xx)d\zz \\
					&+\int_{L_2}^\yy \zz {L_2}^{k-1} \B (\zz,\yy,\xx)d\zz-{\xx}^{k}-\yy{L_2}^{k-1}\\
					&\leq {\xx}^{k-1}\int_{L_1}^\xx \zz \B (\zz,\xx,\yy)d\zz +{L_2}^{k-1} \int_{L_1}^\yy \zz \B (\zz,\yy,\xx)d\zz\\
					&\qquad -{\xx}^{k}-\yy{L_2}^{k-1} \leq 0.
				\end{align*}
				
				\textit{Case 8.} For $(\xx,\yy)\in ({L_2},\infty) \times  ({L_1},{L_2})$, we have  $\Upsilon_\tf(\xx,\yy) = \Upsilon_\tf(\yy,\xx)$ and we deduce from \textit{Case~7} that $\Upsilon_\tf(\xx,\yy)\le 0$.
				
				 \textit{Case~9.} For $(\xx,\yy)\in ({L_2},\infty)^2$, we have $\tf(\xx)= \xx {L_2}^{k-1}$ and $\tf(\yy)=\yy {L_2}^{k-1}$, so that  
				\begin{align*}
					\Upsilon_\tf(\xx,\yy)&=\int_{L_1}^{L_2} {\zz}^{k} \B (\zz,\xx,\yy)d\zz + \int_{L_2}^\xx \zz {L_2}^{k-1} \B (\zz,\xx,\yy)d\zz-\xx {L_2}^{k-1} \\
					& \qquad +\int_{L_1}^{L_2} \zz ^{k}\B (\zz,\yy,\xx)d\zz  +\int_{L_2}^\yy \zz {L_2}^{k-1} \B (\zz,\yy,\xx)d\zz-\yy{L_2}^{k-1}\\
					&\leq {L_2}^{k-1} \int_{L_1}^\xx \zz \B (\zz,\xx,\yy)d\zz +{L_2}^{k-1} \int_{L_1}^\yy \zz \B (\zz,\yy,\xx)d\zz\\
					&\qquad -{\xx} {L_1}^{k-1}-\yy{L_2}^{k-1} \leq 0.
					\end{align*}
		
			\medskip
			
			Since $\Upsilon_\tf(\xx,\yy) \leq 0$ for all $(\xx,\yy)\in (0,\infty)^2$, it follows from~\eqref{eq:wf} that
			\begin{align*}
				\int_{L_1}^{L_2} {\xx}^{k}(\ff (t,\xx)-\ff^{\mbox{\rm{\mbox{in}}}}(\xx))d\xx+ {L_2}^{k-1}\int_{L_2}^\infty \xx(\ff (t,\xx)-\ff^{\mbox{\rm{\mbox{in}}}}(\xx))d\xx\leq 0
			\end{align*}
			for $t\in (0,T)$, which can be rewritten as 
			\begin{align*}
				\int_{L_1}^\infty \min\{{\xx}^{k}, {L_2}^{k-1}\xx\}\ff (t,\xx)d\xx  \leq  \int_{L_1}^\infty {\xx}^{k} \ff^{\mbox{\rm{\mbox{in}}}}(\xx)d\xx
			\end{align*}
            for $t\in (0,T)$. The claimed result follows naturally from allowing ${L_2} \rightarrow \infty$ in the preceding inequality with the help of Fatou's lemma. 
			
		\end{proof}

		\section{Existence by a compactness approach}\label{sec:eca}
		
		This section is devoted to the proof of Theorem~\ref{maintheorem}.   We fix
		\begin{equation} 
			k_0\in (0,1), \qquad p\in (1,1+k_0), \label{eq:fp}
		\end{equation} 
		and consider $\bb$ and $\kk$ satisfying~\eqref{eq:masstransfer}, \eqref{eq:Blocalc}, \eqref{eq:nonintegrable}, and~\eqref{eq:kernel}, respectively. Let $\ff^{\mbox{\rm{\mbox{in}}}} \in  X_{k_0,+} \cap X_{1}$ be an initial condition satsifying $\rho = \mm_1(\ff^{\mbox{\rm{\mbox{in}}}}) > 0$ as well as the additional integrability condition~\eqref{eq:additionalinitiadata}; that is, $\ff^{\mbox{\rm{\mbox{in}}}} \in  X_{1+k_0}$ .

		With a properly constructed estimate for the collision kernel $\kk$, we first demonstrate the well-posedness to \eqref{eq:main} which is based on Picard-Lindel\"of theorem. More specifically,  the truncated collision kernel $\kk_n$ and the initial data $\ff_n^{\mbox{\rm{\mbox{in}}}}$ for integer values of $n\ge 1$ are defined as
		
		\begin{equation}
			\kk_n(\xx, \yy) := \kk(\xx, \yy)\textbf{1}_{(1/n,n)}(\xx)\textbf{1}_{(1/n,n)}(\yy),~ (\xx, \yy) \in (0, \infty)^2, \label{eq:tkernel}
		\end{equation}
		and
		\begin{equation}
			\ff_n^{\mbox{\rm{\mbox{in}}}}:= \ff^{\mbox{\rm{\mbox{in}}}}\textbf{1}_{(0,2n)}. \label{eq:tin}
		\end{equation}

		\begin{prop}\label{prop:locwp}
			For each $n\ge 1$, there is a unique strong solution
			\begin{equation*}
			\ff_n\in C^{1}([0,\infty),X_{k_0,+}\cap X_{1})
			\end{equation*}   
			to
			\begin{equation}
				\partial_{t}\ff_n(t,\xx)=\mathcal{F}_n(\ff_n(t,\xx))-\mathcal{L}_n(\ff_n(t,\xx)), ~~(t,\xx)\in (0,\infty)^2, \label{eq:tmain}
			\end{equation}
			where $\mathcal{F}_n$ and $\mathcal{L}_n$ are defined by~\eqref{eq:formation} and~\eqref{eq:loss}, respectively, with $\kk_n$ instead of $\kk$, and
			\begin{equation}
				\ff_n(0,\xx)=\ff_n^{\mbox{\rm{\mbox{in}}}}(\xx)\geq0, ~~ \xx\in (0,\infty). \label{eq:t_initialdata}
			\end{equation}
			 Moreover, $\ff_n$ satisfies $\ff_n(t,\xx)=0$ for a.e. $\xx > 2n$ and $t\ge0$ and it is also a mass-conserving weak solution to~\eqref{eq:tmain}-\eqref{eq:t_initialdata} on $[0,\infty)$; that is,
			\begin{align}
				\frac{d}{dt}\int_{0}^{\infty}\tf(\xx)&\ff_n(t,\xx)d\xx \nonumber\\ &=\frac{1}{2}\int_{0}^{\infty}\int_{0}^{\infty}\Upsilon_{\tf}(\xx,\yy)\kk_n(\xx,\yy)\ff_n(t,\xx)\ff_n(t,\yy)d\yy d\xx \label{eq:twf}
			\end{align}
			for all $t>0$ and $\tf \in \mathscr{T}^{k_0}$, as well as
			\begin{equation}
				\mm_{1}(\ff_n(t)) = \mm_{1}(\ff_n^{\mbox{\rm{\mbox{in}}}}), \quad t\geq0. \label{mass_truncated_prop}
			\end{equation}
			 Finally, $\ff_n\in  L^\infty((0,\infty),X_{1+k_0})$ with
			\begin{equation}
				\int_{\xx}^{\infty}\yy^{k_0+1}\ff_n(t,\yy)d\yy \le \int_{\xx}^{\infty}\yy^{k_0+1}\ff^{\mbox{\rm{\mbox{in}}}}(\yy)d\yy  \label{tail_truncated_prop}
			\end{equation}
			 for all $(t,\xx)\in [0,\infty) \times (0,\infty)$, which implies in particular that
			\begin{equation}
				\mm_{k_0+1}(\ff_n(t))\leq \mm_{k_0+1}(\ff^{\mbox{\rm{\mbox{in}}}}),\quad t\geq0. \label{k_0+1_truncated_prop}
			\end{equation}
		\end{prop}
		
		\begin{proof}
			First, we will show that  the operator $\mathcal{F}_n-\mathcal{L}_n$ is locally Lipschitz continuous on $X_{k_0}\cap X_{1}$.
			Let $(\ff,\g) \in X_{k_0} \times X_{k_0}$ and observe the following bounds on $\kk_n$:
			\begin{align}
				\kk_n(\xx,\yy) & \leq  4n^{\lambda+k_0} \frac{{\xx}^{k_0} {\yy}^{k_0}}{{\xx}^{k_0}+{\yy}^{k_0}}, \qquad (\xx,\yy)\in (0,\infty)^2, \label{eq:kn_Xk_0_lip} \\
				\kk_n(\xx,\yy) & \leq  4n^{\lambda+1} \frac{\xx \yy }{\xx+\yy},  \qquad (\xx,\yy)\in (0,\infty)^2, \label{eq:kn_X1_lip} \\
				\kk_n(\xx,\yy) &\leq  4n^{\lambda+1} \frac{{\xx}^{k_0} \yy}{{\xx}^{k_0}+{\yy}^{k_0}},  \qquad (\xx,\yy)\in (0,\infty)^2. \label{eq:kn_finite_time_blowup}
			\end{align}

			It follows from~\eqref{eq:masstransfer},  \eqref{eq:Blocalc}, \eqref{eq:nonintegrable}, \eqref{eq:kn_Xk_0_lip}, and the Fubini-Tonelli theorem that
			
			\begin{equation*}
				\begin{split}
					\lvert \lvert \mathcal{F}_n\ff-\mathcal{F}_n\g \rvert \rvert_{k_0}&\leq \frac{1}{2} \int_{0}^{\infty}\int_{\xx}^{\infty}\int_{0}^{\yy }{\xx}^{k_0}\bb (\xx,\yy -\zz,\zz)\kk_n(\yy -\zz,\zz)\\
					& \qquad\qquad \times \lvert \ff (\yy -\zz)\ff (\zz)-\g(\yy -\zz)\g(\zz) \rvert d\zz d\yy d\xx \\
					&=\frac{1}{2} \int_{0}^{\infty}\int_{0}^{\infty}\int_{0}^{\yy+\zz}{\xx}^{k_0}\bb (\xx,\yy,\zz)\kk_n(\yy ,\zz )\\
					& \qquad\qquad\qquad \times \lvert \ff (\yy)\ff (\zz )-\g(\yy)\g(\zz ) \rvert d\xx d\yy d\zz \\
					&\leq \frac{E_{k_0,1}}{2}\int_{0}^{\infty}\int_{0}^{\infty}({\yy}^{k_0}+{\zz }^{k_0})\kk_n(\yy ,\zz )\\
					&\qquad\qquad\qquad \times(\lvert \ff (\yy)\rvert \lvert (\ff-\g)(\zz ) \rvert + \lvert \g(\zz )\rvert \lvert (\ff-\g)(\yy) \rvert ) d\yy d\zz \\
					& \leq  2E_{k_0,1}n^{\lambda+k_0} \int_{0}^{\infty}\int_{0}^{\infty} {\yy}^{k_0}{\zz }^{k_0}\lvert \ff (\yy)\rvert \lvert (\ff-\g)(\zz ) \rvert d\yy d\zz \\
					&\qquad + 2E_{k_0,1}n^{\lambda+k_0} \int_{0}^{\infty}\int_{0}^{\infty} {\yy}^{k_0}{\zz }^{k_0} \lvert \g(\zz )\rvert \lvert (\ff-\g)(\yy) \rvert d\yy d\zz \\
					&\leq  2E_{k_0,1}n^{\lambda+k_0}(\lvert \lvert \ff \rvert \rvert_{k_0}+\lvert \lvert \g \rvert \rvert_{k_0})\lvert \lvert \ff-\g \rvert \rvert_{k_0}.
				\end{split}
			\end{equation*}

			Similarly, using~\eqref{eq:kn_Xk_0_lip}, we estimate
			
			\begin{equation*}
				\begin{split}
					\lvert \lvert \mathcal{L}_n\ff-\mathcal{L}_n\g \rvert \rvert_{k_0} &\leq \int_{0}^{\infty}\int_{0}^{\infty}{\yy}^{k_0}\kk_n(\yy,\zz )\lvert \ff (\yy)\ff (\zz )-\g(\yy)\g(\zz ) \rvert d\yy d\zz \\
					&\leq 4n^{\lambda+k_0} \int_{0}^{\infty}\int_{0}^{\infty}{\yy}^{k_0}{\zz}^{k_0} \lvert \ff (\yy)\rvert \lvert (\ff-\g)(\zz ) \rvert d\yy d\zz \\
					& \qquad + 4n^{\lambda+k_0} \int_{0}^{\infty}\int_{0}^{\infty}{\yy}^{k_0}{\zz}^{k_0} \lvert \g(\zz )\rvert \lvert (\ff-\g)(\yy) \rvert d\yy d\zz \\
					&\leq 4n^{\lambda+k_0}(\lvert \lvert \ff \rvert \rvert_{k_0}+\lvert \lvert \g \rvert \rvert_{k_0})\lvert \lvert \ff-\g \rvert \rvert_{k_0}.
				\end{split}
			\end{equation*}
			
			We have therefore showed the operator $\mathcal{F}_n-\mathcal{L}_n$ is locally Lipschitz continuous on $X_{k_0}$.
			
			Next, let $(\ff,\g) \in X_{1} \times X_{1}$.  It follows from~\eqref{eq:masstransfer},  \eqref{eq:Blocalc}, \eqref{eq:nonintegrable}, \eqref{eq:kn_X1_lip}, and the Fubini-Tonelli theorem that
			
			\begin{align*}
					\lvert \lvert \mathcal{F}_n\ff-\mathcal{F}_n\g \rvert \rvert_{1}&\leq \frac{1}{2} \int_{0}^{\infty}\int_{\xx}^{\infty}\int_{0}^{\yy} \xx\bb (\xx,\yy-\zz ,\zz )\kk_n(\yy-\zz ,\zz )\\
					& \qquad\qquad \times \lvert \ff (\yy-\zz )\ff (\zz )-\g(\yy-\zz )\g(\zz ) \rvert d\zz d\yy d\xx \\
					&=\frac{1}{2} \int_{0}^{\infty}\int_{0}^{\infty}\int_{0}^{\yy+\zz }\xx\bb (\xx,\yy,\zz)\kk_n(\yy,\zz )\\
					&\qquad\qquad\qquad \times \lvert \ff (\yy)\ff (\zz )-\g(\yy)\g(\zz ) \rvert d\xx d\yy d\zz \\
					&\leq \frac{1}{2}\int_{0}^{\infty}\int_{0}^{\infty}(\yy+\zz)\kk_n(\yy,\zz ) \Big[ \lvert \ff (\yy)\rvert \lvert (\ff-\g)(\zz ) \rvert \\
					& \hspace{6cm} + \lvert \g(\zz )\rvert \lvert (\ff-\g)(\yy) \rvert \Big] d\yy d\zz \\
					& \leq  2n^{\lambda+1} \int_{0}^{\infty}\int_{0}^{\infty} \yy \zz \lvert \ff (\yy)\rvert \lvert (\ff-\g)(\zz ) \rvert d\yy d\zz\\
					&\qquad + 2n^{\lambda+1} \int_{0}^{\infty}\int_{0}^{\infty} \yy \zz \lvert \g(\zz )\rvert \lvert (\ff-\g)(\yy) \rvert d\yy d\zz \\
					&\leq  2n^{\lambda+1}(\lvert \lvert \ff \rvert \rvert_{1}+\lvert \lvert \g \rvert \rvert_{1})\lvert \lvert \ff-\g \rvert \rvert_{1}.
			\end{align*}
			
			Similarly, using~\eqref{eq:kn_X1_lip}, we have
			\begin{equation*}
				\begin{split}
					\lvert \lvert \mathcal{L}_n\ff-\mathcal{L}_n\g \rvert \rvert_{1} &\leq \int_{0}^{\infty}\int_{0}^{\infty}\yy\kk_n(\yy,\zz )\lvert \ff (\yy)\ff (\zz )-\g(\yy)\g(\zz ) \rvert d\yy d\xx \\
					&\leq 4n^{\lambda+1} \int_{0}^{\infty}\int_{0}^{\infty}\yy\zz \lvert \ff (\yy)\rvert \lvert (\ff-\g)(\zz ) \rvert d\yy d\zz \\
					&\qquad + 4n^{\lambda+1} \int_{0}^{\infty}\int_{0}^{\infty}\yy\zz \lvert \g(\zz )\rvert \lvert (\ff-\g)(\yy) \rvert )d\yy d\zz \\
					&\leq 4n^{\lambda+1}(\lvert \lvert \ff \rvert \rvert_{1}+\lvert \lvert \g \rvert \rvert_{1})\lvert \lvert \ff-\g \rvert \rvert_{1}.
				\end{split}
			\end{equation*}
			Thus, we have shown that $\mathcal{F}_n-\mathcal{L}_n$ is locally Lipschitz continuous on $X_1$.
			
			Introducing the operator
			\begin{equation}
				\Bar{\mathcal{F}_n} \ff = \big( \mathcal{F}_n \ff \big)_{+} = \max\big\{ \mathcal{F}_n \ff , 0 \big\}, \qquad u\in X_{k_0}\cap X_1, \label{eq:tmain_+veG}
			\end{equation}
			it is also a locally Lipschitz continuous map from $X_{k_0} \cap X_1$ to $X_{k_0} \cap X_1$. Consequently, the Picard-Lindel\"of theorem guarantees that there are  $T_n\in(0,\infty]$ and a unique function $\ff_n\in C^{1}([0,T_n),X_{k_0} \cap X_1)$ such that $\ff_n$ solves 
			\begin{equation}
				\partial_{t}\ff_n(t,\xx)=\bar{\mathcal{F}_n}(\ff_n(t,\xx))-\mathcal{L}_n(\ff_n(t,\xx))~ \text{\mbox{\rm{\mbox{in}}}} ~ (0,T_n)\times (0,\infty) \label{eq:treg}
			\end{equation}
			with initial condition $\ff_n(0,\cdot)=\ff_n^{\mbox{\rm{\mbox{in}}}}$, recalling that the latter is defined in~\eqref{eq:tin}. In addition, there are two possibilities:  either $T_n=\infty$ or $T_n<\infty$ and $(\lvert \lvert \ff (t) \rvert \rvert_{k_0} +\lvert \lvert \ff (t) \rvert \rvert_1)\rightarrow \infty $ as $t \rightarrow T_n$.
			
			Now, it follows from~\eqref{eq:kn_Xk_0_lip}, \eqref{eq:tmain_+veG}, and~\eqref{eq:treg} that, for $t\in (0,T_n)$,
			\begin{equation*}
				\begin{split}
					& \hspace{-0.5cm}  \frac{d}{dt}\int_{0}^{\infty}{\xx}^{k_0}(-\ff_n(t,\xx))_{+}d\xx \\
					&= -\int_{0}^{\infty}{\xx}^{k_0}sign_{+}(-\ff_n(t,\xx))\partial_{t}\ff_n(t,\xx)d\xx \\
					&\leq \int_{0}^{\infty}{\xx}^{k_0}sign_{+}(-\ff_n(t,\xx))\ff_n(t,\xx) \int_{0}^{\infty}\kk_n(\xx,\yy)\ff_n(t,\yy)d\yy d\xx \\
					& \leq \int_{0}^{\infty}{\xx}^{k_0}(-\ff_n(t,\xx))_{+} \int_{0}^{\infty}\kk_n(\xx,\yy)\lvert \ff_n(t,\yy)\rvert  d\yy d\xx \\
					& \leq  4n^{\lambda+k_0}\lvert \lvert \ff_n(t) \rvert \rvert_{k_0} \lvert \lvert (-\ff_n(t))_{+} \rvert \rvert_{k_0}.
				\end{split}
			\end{equation*}
			 Hence, after integrating with respect to time,
			\begin{equation*}
				\lvert \lvert (-\ff_n(t))_{+} \rvert \rvert_{k_0} \leq \lvert \lvert (-\ff_n^{\mbox{\rm{\mbox{in}}}})_{+} \rvert \rvert_{k_0}  \exp \{ 4n^{\lambda+k_0}\int_{0}^{t}\lvert \lvert \ff_n(s) \rvert \rvert_{k_0}  ds \}=0,
			\end{equation*}
			which shows that $\ff_n (t) \in X_{k_0,+}$ for all $t\in (0,T_n)$. The non-negativity of $\ff_n$ readily implies that $ \Bar{\mathcal{F}_n}\ff_n= {\mathcal{F}_n}\ff_n$ and we deduce from~\eqref{eq:treg} that $\ff_n\in C^{1}([0,T_n),X_{k_0} \cap X_1)$ solves \eqref{eq:tmain}-\eqref{eq:t_initialdata}.
			
			Next, it follows from \eqref{eq:kn_Xk_0_lip} that
			\begin{align*}
				\int_{0}^{\infty} \int_{0}^{\infty}({\xx}^{k_0}+{\yy}^{k_0})\kk_n(\xx,\yy)\ff_n(\tau,x)\ff_n(\tau,y)d\yy d\xx  & \leq  4n^{\lambda+k_0}\lvert \lvert \ff_n(\tau) \rvert \rvert^2_{k_0} \\
				&\leq 4n^{\lambda+k_0} \sup_{\tau\in [0,t]}\{\lvert \lvert \ff_n(\tau) \rvert \rvert^2_{k_0} \}
			\end{align*}
			for $\tau \in (0, t)$ and $t \in (0, T_n)$, so that $\ff_n$ satisfies~\eqref{eq:ws2}. Thus, $\ff_n$ is a weak solution to \eqref{eq:tmain}-\eqref{eq:t_initialdata} on $[0, T_n)$. 
			
			 Next, it follows from~\eqref{eq:kn_X1_lip} that, for any $t \in (0, T_n)$ and $\tau \in (0, t)$, the following inequality holds:
			\begin{align*}
				\int_{0}^{\infty} \int_{0}^{\infty}(\xx+\yy)\kk_n(\xx,\yy)\ff (\tau,\xx)\ff (\tau,\yy)d\xx d\yy \leq 4n^{\lambda+1} \sup_{\tau\in [0,t]}{\lvert \lvert \ff_n(\tau) \rvert \rvert_{1}^2}.
			\end{align*}
			Thus, based on Proposition~\ref{prop:massconservation} applied with $\bar{\Phi}=\Phi_n \le \Phi$, it follows that $\ff_n$ is a mass-conserving weak solution to~\eqref{eq:tmain}-\eqref{eq:t_initialdata} on the interval $[0, T_n)$. Therefore,
			\begin{equation}
				\lvert \lvert \ff_n(t) \rvert \rvert_1= \mm_1(\ff_n(t))=\mm_1(\ff_n^{\mbox{\rm{\mbox{in}}}})=  \lvert \lvert \ff_n^{\mbox{\rm{\mbox{in}}}}(t) \rvert \rvert_1\leq \rho , ~ t\in [0,T_n). \label{eq:tmassconservation}
			\end{equation}
			
		 Also, let $t\in (0,T_n)$. Substituting $\tf(\xx)=W_{k_0}(\xx) = \xx^{k_0}$ into~\eqref{eq:twf} and using~\eqref{eq:nonintegrable}, \eqref{eq:kn_finite_time_blowup},  and~\eqref{eq:tmassconservation} , we get
		\begin{equation*}
			\begin{split}
				\frac{d}{dt}\lvert \lvert \ff_n(t) \rvert \rvert_{k_0}  &\leq \frac{1}{2}\int_{0}^{\infty}\int_{0}^{\infty}\Upsilon_{W_{k_0}}(\xx,\yy)\kk_n(\xx,\yy)\ff_n(t,\xx)\ff_n(t,\yy)d\yy d\xx \\
				& \leq \frac{E_{k_0,1}}{2}\int_{0}^{\infty}\int_{0}^{\infty}({\xx}^{k_0} +{\yy}^{k_0})\kk_n(\xx,\yy)\ff_n(t,\xx)\ff_n(t,\yy)d\yy d\xx \\
				& \leq 2E_{k_0,1}n^{\lambda+2k_0+1}\int_{0}^{\infty}\int_{0}^{\infty}{\xx}^{k_0} \yy\ff_n(t,\xx)\ff_n(t,\yy)d\yy d\xx \\
				&\leq 2E_{k_0,1}n^{\lambda+1} \lvert \lvert \ff_n(t) \rvert \rvert_{k_0} \lvert \lvert \ff_n(t) \rvert \rvert_{1} \\
				&\leq 2E_{k_0,1}n^{\lambda+1} \rho \lvert \lvert \ff_n(t) \rvert \rvert_{k_0}.
			\end{split}
		\end{equation*}
		 Consequently, 
			\begin{equation*}
				\lvert \lvert \ff_n(t) \rvert \rvert_{k_0} \le \lvert \lvert \ff_n^{\mbox{\rm{\mbox{in}}}} \rvert \rvert_{k_0} \exp\{ 2 E_{k_0,1} n^{\lambda+1} \rho t\}, \qquad t\in [0,T_n),
			\end{equation*}
		which implies, together with the mass conservation property~\eqref{eq:tmassconservation}, that $(\lvert \lvert \ff (t) \rvert \rvert_{k_0} +\lvert \lvert \ff (t) \rvert \rvert_1)\nrightarrow \infty $  as $t \rightarrow T_n$. Hence $T_n=\infty$. 
		
		\medskip
		
		Finally, since $\ff^{\mbox{\rm{\mbox{in}}}}$ satisfies~\eqref{eq:additionalinitiadata}, by Lemma~\ref{lem:mtail}, we have
			\begin{equation*}
				\int_{\xx}^{\infty}\yy^{k_0+1}\ff_n(t,\yy)d\yy\le \int_{\xx}^{\infty}\yy^{k_0+1}\ff_n^{\mbox{\rm{\mbox{in}}}}(\yy)d\yy\le \int_{\xx}^{\infty}\yy^{k_0+1}\ff^{\mbox{\rm{\mbox{in}}}}(\yy)d\yy 
			\end{equation*}
			 for all $(t,x_1)\in [0,\infty)\times (0,\infty)$. In particular, choosing $x_1=0$ leads to the estimate 
				\begin{equation*}
					\mm_{k_0+1}(\ff_n(t))\leq \mm_{k_0+1}(\ff^{\mbox{\rm{\mbox{in}}}}) , ~ t\in [0,\infty),
				\end{equation*}
				while the choice $\xx = 2n$ and~\eqref{eq:tin} implies that
			\begin{equation*}
				\int_{2n}^{\infty}\yy^{k_0+1}\ff_n(t,\yy)d\yy\le \int_{2n}^{\infty}\yy^{k_0+1}\ff^{\mbox{\rm{\mbox{in}}}}(\yy)d\yy=0 ,\quad t\in [0,\infty).
			\end{equation*}
			Hence, owing to the non-negativity of $\ff_n$,
			\begin{equation*}
				\ff_n(t,\xx) = 0~\text{ for a.e. }~ \xx>2n~\text{and}~ t\in [0,\infty),
			\end{equation*}
			{and the proof of \Cref{prop:locwp} is complete.}
		\end{proof}
		
		 Having proved the well-posedness of~\eqref{eq:tmain}-\eqref{eq:t_initialdata} for each $n\ge 1$, we now turn to the derivation of estimates which do not depend on $n\ge 1$, in order to be able to find cluster points of the sequence $(u_n)_{n\ge 1}$ as $n\to\infty$, which are natural candidates as solutions to~\eqref{eq:main}-\eqref{eq:in} according to the choice~\eqref{eq:tkernel} of $\Phi_n$. Recalling~\eqref{mass_truncated_prop} and~\eqref{k_0+1_truncated_prop}, we have already obtained estimates on $(u_n)_{n\ge 1}$ for large size particles which do not depend on $n\ge 1$. The next step is to look for a similar estimate for small size particles.
		
		\subsection{Control on small size particles}
		We now study the behaviour for small sizes. To this end, we need the following lemma. 
		
		\begin{lem}\label{thm:smallsizebehavious}
			 Let $T_*$ be defined in \Cref{maintheorem}, $T \in (0, T_*)$ and $C_0 > 0$ such that $\mm_{k_{0}}(\ff^{\mbox{\rm{\mbox{in}}}})\leq C_0$.  There exists a positive constant $C_1(T)$ which is independent of $n\geq 1$ and depends solely on $\kk$, $\beta$, $\ff^{\mbox{\rm{\mbox{in}}}}$, $C_0$, and $T$, such that the following inequality holds for all $t\in [0,T]$,
			\begin{align*}
				\mm_{k_0}(\ff_n(t))\leq C_1(T).
			\end{align*}
		\end{lem}
		
		\begin{proof}
			For $t\ge 0$ and $\tf(\xx)=W_{k_0}(\xx) = \xx^{k_0}$, we deduce from~\eqref{eq:nonintegrable}, \eqref{eq:kernel}, and~\eqref{eq:twf} that
			\begin{align}
					\frac{d}{dt}\mm_{k_0}(\ff_n(t)) &\leq \frac{1}{2}\int_{0}^{\infty} \int_{0}^{\infty}\Upsilon_{\tf}(\xx,\yy)\kk(\xx,\yy)\ff_n(t,\xx)\ff_n(t,\yy)d\yy d\xx \nonumber \\
					&\leq E_{k_0,1}\int_{0}^{\infty} \int_{0}^{\infty}({\xx}^{k_0}+{\yy}^{k_0}){\xx}^{\lambda_1} \yy ^{\lambda_2} \ff_n(t,\xx)\ff_n(t,\yy)d\yy d\xx \nonumber \\
					&\leq E_{k_0,1}[\mm_{k_0+{\lambda_1}}(\ff_n(t))\mm_{\lambda_2}(\ff_n(t))+\mm_{k_0+{\lambda_2}}(\ff_n(t))\mm_{\lambda_1} (\ff_n(t))]. \label{eq:differetial_small size}
			\end{align}
			As $(\lambda_1,\lambda_2)\in [k_0,1]^2$, we obtain the following moment estimates by means of H\"older's inequality and~\eqref{mass_truncated_prop}
			\begin{equation*}
				\mm_{\lambda_1} (\ff_n)\leq \mm_{k_0} (\ff_n)^{\frac{1-{\lambda_1}}{1-k_0}}\mm_1(\ff_n)^{\frac{{\lambda_1}-k_0}{1-k_0}}\leq \rho^{\frac{{\lambda_1}-k_0}{1-k_0}}\mm_{k_0} (\ff_n)^{\frac{1-{\lambda_1}}{1-k_0}},
			\end{equation*}
			and
			\begin{equation*}
				\mm_{\lambda_2} (\ff_n)\leq \mm_{k_0} (\ff_n)^{\frac{1-{\lambda_2}}{1-k_0}}\mm_1(\ff_n)^{\frac{{\lambda_2}-k_0}{1-k_0}}\leq \rho^{\frac{{\lambda_2}-k_0}{1-k_0}}\mm_{k_0} (\ff_n)^{\frac{1-{\lambda_2}}{1-k_0}}.
			\end{equation*}
			
			In addition, when $k_0+{\lambda_1} \in [k_0,1]$, we have 
			\begin{equation*}
				\mm_{k_0+{\lambda_1}}(\ff_n)\leq \mm_{k_0}(\ff_n)^{\frac{1-k_0-{\lambda_1}}{1-k_0}}\mm_1(\ff_n)^{\frac{{\lambda_1}}{1-k_0}}\leq \rho^{\frac{{\lambda_1}}{1-k_0}} \mm_{k_0}(\ff_n)^{\frac{1-k_0-{\lambda_1}}{1-k_0}},
			\end{equation*}
			and when  $k_0+1 \geq k_0+{\lambda_1}\geq 1$, we get
			\begin{align*}
				\mm_{k_0+{\lambda_1}}(\ff_n)&\leq \mm_1(\ff_n)^{\frac{1-{\lambda_1}}{k_0}}\mm_{k_0+1}(\ff_n)^{\frac{k_0+{\lambda_1}-1}{k_0}}\leq \rho^{\frac{1-{\lambda_1}}{k_0}} \mm_{k_0+1}(\ff_n)^{\frac{k_0+{\lambda_1}-1}{k_0}}\\
				&\leq \rho^{\frac{1-{\lambda_1}}{k_0}} \mm_{k_0+1}(\ff^{\mbox{\rm{\mbox{in}}}})^{\frac{k_0+{\lambda_1}-1}{k_0}},
			\end{align*}
			 where we have used~\eqref{k_0+1_truncated_prop} to obtain the last estimate.

			Thus
			\begin{equation*}\mm_{k_0+{\lambda_1}}(\ff_n)\leq c_1 \mm_{k_0}(\ff_n)^{\frac{(1-k_0-{\lambda_1})_+}{1-k_0}},\end{equation*}
			with 
			
			\begin{align*}
				c_1:=\max\left\{ \rho^{\frac{{\lambda_1}}{1-k_0}}, \rho^{\frac{1-{\lambda_1}}{k_0}} \mm_{k_0+1}(\ff^{\mbox{\rm{\mbox{in}}}})^{\frac{k_0+{\lambda_1}-1}{k_0}} \right\}>0.
			\end{align*}
			
			Similarly,
			\begin{equation*} \mm_{k_0+{\lambda_2}}(\ff_n)\leq c_2 \mm_{k_0}(\ff_n)^{\frac{(1-k_0-{\lambda_2})_+}{1-k_0}},\end{equation*}
			with 
			\begin{equation*}
				c_2:=\max\left\{ \rho^{\frac{{\lambda_2}}{1-k_0}}, \rho^{\frac{1-{\lambda_2}}{k_0}} \mm_{k_0+1}(\ff^{\mbox{\rm{\mbox{in}}}})^{\frac{k_0+{\lambda_2}-1}{k_0}} \right\}>0.
			\end{equation*}
		Inserting the above estimates on $\mm_{\lambda_i}$ and $\mm_{k_0+\lambda_i}$ for $i=\{1,2\}$ into the inequality~\eqref{eq:differetial_small size},  we get
			\begin{equation}
				\frac{d}{dt}\mm_{k_0}(\ff_n(t)) \leq  {c_3}\bigg[\mm_{k_0}(\ff_n)^{\frac{1-{\lambda_2} +(1-k_0-{\lambda_1})_+}{1-k_0}}+\mm_{k_0}(\ff_n)^{\frac{1-{\lambda_1} +(1-k_0-{\lambda_2})_+}{1-k_0}}\bigg], \label{eq:differential_moment}
			\end{equation}
			with 
			\begin{equation*}
				 c_3 :=\max\left\{ c_1\rho^{\frac{{\lambda_2}-k_0}{1-k_0}}, c_2\rho^{\frac{{\lambda_1}-k_0}{1-k_0}} \right\}>0. 
			\end{equation*}
			
			 Now, the estimate to be deduced from~\eqref{eq:differential_moment} obviously depends on the range of the powers of $\mm_{k_0}(\ff_n)$ in the right-hand side of~\eqref{eq:differential_moment}, which  requires to split the analysis. 
				
			\begin{enumerate}
				\item If $\lambda={\lambda_1}+{\lambda_2} \in  [2k_0,1)$, then 
				\begin{equation*}k_0+{\lambda_1}\leq k_0+{\lambda_2} \leq {\lambda_1}+{\lambda_2}<1,\end{equation*}
				which implies that 
				\begin{equation*}
					\frac{1-{\lambda_1}+(1-k_0-{\lambda_2})_+}{1-k_0}=\frac{1-k_0+1-\lambda}{1-k_0}>1
				\end{equation*}
				and
				\begin{equation*}
					\frac{1-{\lambda_2}+(1-k_0-{\lambda_1})_+}{1-k_0}=\frac{1-k_0+1-\lambda}{1-k_0}>1.
				\end{equation*}
				 In that case, the right-hand side of the differential inequality~\eqref{eq:differential_moment} is a superlinear function of $\mm_{k_0}(\ff_n)$ and integrating~\eqref{eq:differential_moment} gives
				\begin{align*}
					\mm_{k_0}(\ff_n(t)) & \le \left[ \mm_{k_0}(\ff_n^{\mbox{\rm{\mbox{in}}}})^{-\frac{1-\lambda}{1-k_0}} - \frac{2(1-\lambda)c_3}{1-k_0} t \right]^{-\frac{1-k_0}{1-\lambda}} \\
					& \le \left[ \mm_{k_0}(\ff^{\mbox{\rm{\mbox{in}}}})^{-\frac{1-\lambda}{1-k_0}} - \frac{2(1-\lambda)c_3}{1-k_0} t \right]^{-\frac{1-k_0}{1-\lambda}},
				\end{align*}	
				provided $t\in [0,T_{k_0})$ with
				\begin{equation}
					T_{k_0} := \frac{1-k_0}{2(1-\lambda) c_3} \mm_{k_0}(\ff^{\mbox{\rm{\mbox{in}}}})^{-\frac{1-\lambda}{1-k_0}} . \label{eq:Tk0}
				\end{equation}
				We have thus shown that, for $T\in (0,T_{k_0})$ and $t\in [0,T]$, 
				\begin{equation*}
					\mm_{k_0}(\ff_n(t)) \le C_1(T) := \left[ \mm_{k_0}(\ff^{\mbox{\rm{\mbox{in}}}})^{-\frac{1-\lambda}{1-k_0}} - \frac{2(1-\lambda)c_3}{1-k_0} T \right]^{-\frac{1-k_0}{1-\lambda}},
				\end{equation*}
			which completes the proof of \Cref{thm:smallsizebehavious} for $\lambda\in [2k_0,1)$.
				
				\item If $\lambda={\lambda_1}+{\lambda_2} \in [1,2]$,  then 
				\begin{equation*}
					\frac{1-{\lambda_1}}{1-k_0}+\frac{(1-k_0-{\lambda_2})_+}{1-k_0} = \left\{
					\begin{array}{lcl}
						\displaystyle{\frac{1-{\lambda_1}}{1-k_0}} \leq 1 & \text{if} & k_0+{\lambda_2}\geq 1,\\
						\displaystyle{1+\frac{1-\lambda}{1-k_0}}\leq 1 & \text{if} & k_0+{\lambda_2} < 1,
					\end{array} 
					\right.
				\end{equation*}
				and, similarly, 
				\begin{equation*}
					\frac{1-{\lambda_2}+(1-k_0-{\lambda_1})_+}{1-k_0}\leq 1.
				\end{equation*}
				 Consequently, by Young's inequality and~\eqref{eq:differential_moment},
					\begin{equation*}
						\frac{d}{dt}\mm_{k_0}(\ff_n(t)) \leq \frac{2c_3}{1-k_0} \big[ 1 + \mm_{k_0}(\ff_n(t))  \big],
					\end{equation*}
				from which we readily obtain that 
				\begin{equation*}
					\mm_{k_0}(\ff_n(t))\le \bigg[ 1 + \mm_{k_0}(\ff_n^{\mbox{\rm{\mbox{in}}}}) \bigg] e^{\frac{2c_3t}{1-k_0} } \le \bigg[ 1 + \mm_{k_0}(\ff^{\mbox{\rm{\mbox{in}}}}) \bigg] e^{\frac{2c_3t}{1-k_0} }
				\end{equation*}
				for $t\ge 0$. Introducing 
				\begin{equation*}
					C_1(T) := \bigg[ 1 + \mm_{k_0}(\ff^{\mbox{\rm{\mbox{in}}}}) \bigg] e^{\frac{2c_3T}{1-k_0} }
				\end{equation*}
				for $T>0$, we obtain the estimate stated in \Cref{thm:smallsizebehavious}.
		
			\end{enumerate}
		\end{proof}

		Next, we derive a more precise approximation for small sizes. For that purpose, we recall the following variant of the de la Vall\'ee-Poussin theorem \cite{dlvp1915} which we apply to $\ff^{\mbox{\rm{\mbox{in}}}} \in X_{k_0}$ and can be established as  \cite[Lemma~A.1]{GL12021}.
		
		\begin{prop}\label{lem:davl}
			Recalling that $k_0\in (0,1)$ and $p\in (1,1+k_0)$ are defined in~\eqref{eq:fp}	, we pick $\epsilon\in \big(-k_0,(p-k_0-1)/p\big)$. There exists a non-negative, convex, and non-increasing function $\psi_{0}\in C^{1}((0,\infty))$ such that
			\begin{equation}
				\mm_{\psi_0}(\ff^{\mbox{\rm{\mbox{in}}}}) := \int_{0}^{\infty}\psi_{0}(\xx) \lvert \ff^{\mbox{\rm{\mbox{in}}}}(\xx) \rvert d\xx  < \infty \label{eq:psi0in}
			\end{equation}
			and 
			\begin{align}\label{eq:monotonocity_psi_0}
				\lim_{x \to 0} \frac{\psi_{0}(\xx)}{{\xx}^{k_0}}=\infty, ~ ~ ~ ~  ~\lim_{x \to 0} {\xx}^{\epsilon}\psi_{0}(\xx)=0, ~ ~ ~\xx \mapsto {\xx}^{\epsilon}\psi_{0}(\xx) \text{is non-decreasing.}
			\end{align}
		\end{prop}

		\begin{lem}\label{ssbrevisted}
			 For $T \in (0, T_*)$, there exists a positive constant $C_2(T)$ which is independent of $n\geq 1$ and depends solely on $\kk$, $\beta$, $\ff^{\mbox{\rm{\mbox{in}}}}$, and $T$, such that the following inequality holds for all $t\in [0,T]$,
			\begin{equation}
				\mm_{\psi_0}(\ff_n(t)):=\int_{0}^{\infty}\psi_{0}(\xx)\ff_n(t,\xx)d\xx  \leq C_2(T). \label{eq:small_size_revisited_truncated}
			\end{equation}
		\end{lem}

		\begin{proof}
			We first observe that, by Young's inequality and~\eqref{eq:kernel},
			\begin{equation}
				\kk(\xx,\yy)\leq 2({\xx}^{k_0} +\xx)({\yy}^{k_0} +\yy), \qquad (\xx,\yy)\in (0,\infty)^2. \label{eq:K_uniform_bound}
			\end{equation}
			
			Now, let $t\in [0,T]$. By~\eqref{eq:masstransfer}, \eqref{eq:Blocalc}, \eqref{eq:zeta},  \eqref{eq:K_uniform_bound}, and the non-negativity of $\psi_0$, 
			
			\begin{align}
					\frac{d}{dt}\mm_{\psi_0}(\ff_n(t)) &\leq \frac{1}{2}\int_{0}^{\infty} \int_{0}^{\infty}\Upsilon_{\psi_{0}}(\xx,\yy)\kk(\xx,\yy)\ff_n(t,\xx)\ff_n(t,\yy)d\yy d\xx \nonumber \\
					&\hspace{-1cm} \leq \int_{0}^{\infty} \int_{0}^{\infty}  \int_{0}^{\xx}\psi_{0}(\zz ) \B (\zz,\xx,\yy)({\xx}^{k_0} +\xx)({\yy}^{k_0} +\yy)  \nonumber \\
					&  \hspace{5cm} \times\ff_n(t,\xx)  \ff_n(t,\yy)d\zz d\yy d\xx \nonumber \\
					& \hspace{-0.5cm} +\int_{0}^{\infty} \int_{0}^{\infty}  \int_{0}^{\yy}\psi_{0}(\zz ) \B (\zz,\yy,\xx) ({\xx}^{k_0} +\xx) ({\yy}^{k_0} +\yy) \nonumber \\
					& \hspace{5cm} \times \ff_n(t,\xx)\ff_n(t,\yy)d\zz d\yy d\xx. \label{eq:diffrential_more_precise_small_size}
			\end{align}
			
			Next, from~\eqref{eq:nonintegrable},  \eqref{eq:monotonocity_psi_0} (with the above choice of $p$ and $\epsilon$), and H\"older's inequality, 
			\begin{align*}
					\int_{0}^{\xx}\psi_{0}(\zz ) \B (\zz,\xx,\yy)d\zz  &=\int_{0}^{\xx}{\zz}^\epsilon \psi_{0}(\zz ) {\zz}^{\big(-\epsilon-\frac{k_0}{p}\big)}{\zz}^{\frac{k_0}{p}}\B (\zz,\xx,\yy)d\zz \\
					& \leq {\xx}^\epsilon \psi_{0}(\xx) \Big(\int_{0}^{\xx}{\zz}^{\frac{(-k_0- p \epsilon)}{p-1}}d\zz \Big)^{\frac{p-1}{p}} \\
					& \hspace{2cm} \times  \Big(\int_{0}^{\xx} {\zz}^{k_0}\B (\zz,\xx,\yy)^{p}d\zz \Big)^{\frac{1}{p}}\\
					& \leq B_{k_0,p}\psi_{0}(\xx),
			\end{align*}
			where 
			\begin{equation*} 
			B_{k_0,p}:=\bigg[\bigg(\frac{p-1}{p-1-\epsilon p-k_0}\bigg)^{(p-1)}E_{k_0,p}\bigg]^{\frac{1}{p}}.
			\end{equation*}
			
			Now we deduce from the above inequality and~\eqref{eq:diffrential_more_precise_small_size} that
			\begin{equation*}
				\begin{split}
					\frac{d}{dt}\mm_{\psi_0}(\ff_n(t))
					&\leq  2 B_{k_0,p}\int_{0}^{\infty} \int_{0}^{\infty}\psi_{0}(\xx)({\yy}^{k_0} +\yy)({\xx}^{k_0}+\xx)\ff_n(t,\xx)\ff_n(t,\yy)d\yy d\xx.
				\end{split}
			\end{equation*}
			 Since
			\begin{align*} 
					\int_{0}^{\infty}\psi_{0}(\xx)({\xx}^{k_0}+\xx)\ff_n (t,\xx)d\xx &\leq 2\int_{0}^{1}\psi_0(\xx)\ff_n(t,\xx)d\xx\\
					& +  2 \psi_0(1)\int_{1}^{\infty} \xx\ff_n(t,\xx)d\xx \\
					&\leq 2\mm_{\psi_0}(\ff_n(t))+2\psi_0(1)\rho
			\end{align*}
			 by~\eqref{mass_truncated_prop} and~\Cref{lem:davl}, we further obtain, using again~\eqref{mass_truncated_prop}, along with \Cref{thm:smallsizebehavious},
			\begin{align*}
					\frac{d}{dt}\mm_{\psi_0}(\ff_n(t))
					&\leq  4  B_{k_0,p} [\mm_{\psi_0}(\ff_n(t)) + \psi_0(1)\rho] \left[ \rho + \mm_{k_0}(\ff_n(t)) \right]\\
					&\leq 2B_{k_0,p}[\mm_{\psi_0}(\ff_n(t))+\psi(1)\rho][\rho+C_1(T)].
			\end{align*}
			Integrating the above inequality with respect to time gives the desired result.
			\end{proof}

		\subsection{Uniform integrability} 
		
		 Having derived estimates on the sequence $(\ff_n)_{n\ge 1}$ for large and small sizes, we shall now discuss the uniform integrability of this sequence in a suitably chosen weighted $L^1$-space. In that direction the main tool is another variant of the de la Vall\'ee Poussin theorem \cite[Theorem~7.1.6]{bll2019} which we apply here to $\ff^{\mbox{\rm{\mbox{in}}}} \in X_{k_0}$. It guarantees that there exists a function $\psi \in C^1([0, \infty))$ satisfying the following properties: 
		$\psi$ is a convex function with $\psi(0) = \psi^{\prime}(0) = 0$ and $\psi^{\prime}$ is concave and positive on $(0, \infty)$,
		\begin{align}
			\int_{0}^{\infty}{\xx}^{k_0}\psi(\ff^{\mbox{\rm{\mbox{in}}}}(\xx))d\xx  <\infty,\\
			\lim_{s \to \infty}\psi^{\prime}(s)=\lim_{s \to \infty}\frac{\psi(s)}{s}=\infty, \label{eq:monotonocity_psi}
		\end{align}
		and, for all $r \in (1,2]$,
		\begin{align}
			\lim_{s \to \infty}\frac{\psi^{\prime}(s)}{s^{r-1}}=\lim_{s \to \infty}\frac{\psi(s)}{s^r}=0,
		\end{align}
		which also guarantees that 		
		\begin{equation}
			B_{r}:=\sup_{s\geq 0}\Big\{\frac{\psi(s)}{s^{r}}\Big\} <\infty. \label{eq:Bp}
		\end{equation}
		Also, we recall the following properties of the convex function $\psi$ (see \cite[Proposition~7.1.9~(a) \&~(b)]{bll2019}), 
		\begin{align}
			& \psi(s) \le s\psi^\prime(s) \le 2\psi(s), \qquad s \ge 0, \label{eq:psi1} \\
			& r\psi^\prime(s)\le \psi(s)+\psi(r), \qquad (r,s)\in [0,\infty)^2. \label{eq:psi2}
		\end{align}
		In particular, it follows from~\eqref{eq:psi2} and the non-negativity of $\ff_n$ and $\bb$ that
		\begin{equation}
			\psi^{\prime}(\ff_n(t,\zz))\bb (\zz,\xx,\yy)\leq \psi(\ff_n(t,\zz))+\psi(\bb (\zz,\xx,\yy)),~~ \zz \in(0,\xx+\yy), \label{eq:properties_of_psi}
		\end{equation}
		for all  $(\xx,\yy)\in (0,\infty)^2$.
		
		\begin{lem}\label{lem:uniform_int}  For $T \in (0,T_*)$, there exists a positive constant $C_3(T)$, which is independent of $n\geq 1$ and depends solely on $\kk$, $b$, $\ff^{\mbox{\rm{\mbox{in}}}}$, and $T$, such that the following inequality holds for all $t\in [0,T]$,
			\begin{equation*}
				\mm_{\xi}(\psi(\ff_n(t))) \leq C_3(T),
			\end{equation*}
			where $\xi(\xx)=\min\{\xx,{\xx}^{k_0}\}$ for $\xx\in (0,\infty)$.
		\end{lem}

		\begin{proof}
			Let $t\in [0,T]$. Thanks to~\eqref{eq:kernel}, \eqref{eq:twf}, \eqref{eq:K_uniform_bound}, \eqref{eq:properties_of_psi}, and the non-negativity of both $\xi$, $\ff_n$ and $\psi^\prime$, we find
			
			\begin{align}
					\frac{d}{dt}\mm_{\xi}(\psi(\ff_n(t))) & \leq \frac{1}{2} \int_{0}^{\infty} \int_{0}^{\infty}\int_{0}^{\xx+\yy}\xi(\zz ) \psi^{\prime}(\ff_n(t,\zz ))\bb (\zz,\xx,\yy)\kk(\xx,\yy) \nonumber\\
					& \hspace{6cm} \times \ff_n(t,\xx)\ff_n(t,\yy)d\zz d\yy d\xx \nonumber \\
					&\leq \int_{0}^{\infty} \int_{0}^{\infty}\int_{0}^{\xx+\yy}\xi(\zz ) \psi(\ff_n(t,\zz))(\xx+{\xx}^{k_0})(\yy+{\yy}^{k_0})  \nonumber \\
					& \hspace{6cm}\times \ff_n(t,\xx)\ff_n(t,\yy)d\zz d\yy d\xx \nonumber \\
					& \qquad + \int_{0}^{\infty} \int_{0}^{\infty}\int_{0}^{\xx+\yy}\xi(\zz ) \psi(\bb (\zz,\xx,\yy)) {\xx}^{\lambda_1} {\yy}^{\lambda_2} \nonumber\\
					& \hspace{6cm} \times \ff_n(t,\xx)\ff_n(t,\yy)d\zz d\yy d\xx.	 \label{eq:differential_uniform_integrability}
			\end{align}
			
			We now estimate the following integral with the help of~\eqref{eq:masstransfer}, \eqref{eq:Blocalc}, \eqref{eq:nonintegrable}, \eqref{eq:fp}, and~\eqref{eq:Bp} 
			\begin{align*}
				\int_{0}^{\xx+\yy}\xi(\zz )\psi(\bb (\zz,\xx,\yy))d\zz &= \int_{0}^{\xx+\yy}\xi(\zz )\frac{\psi(\bb (\zz,\xx,\yy))}{\bb (\zz,\xx,\yy)^{p}}\bb (\zz,\xx,\yy)^{p}d\zz \\
				&\leq B_p\int_{0}^{\xx+\yy}\xi(\zz )\bb (\zz,\xx,\yy)^{p}d\zz \\
				& \leq B_p\int_{0}^{\xx}\xi(\zz )\B (\zz,\xx,\yy)^{p}d\zz \\
				&\qquad +B_p\int_{0}^{\yy}\xi(\zz )\B (\zz,\yy,\xx)^{p}d\zz .
			\end{align*}
			Now, either $\xx\in(0,1)$ and 
			\begin{align*}
				\int_{0}^{\xx}\xi(\zz )\B (\zz,\xx,\yy)^{p}d\zz &\leq\int_{0}^{\xx}\zz \B (\zz,\xx,\yy)^{p}d\zz\\
				& \leq  \int_{0}^{\xx} {\zz}^{k_0}\B (\zz,\xx,\yy)^{p}d\zz \\
				& \leq E_{k_0,p}  {\xx}^{k_0+1-p},
			\end{align*}
			or $\xx\in(1,\infty)$ and
			\begin{align*}
				\int_{0}^{\xx}\xi(\zz )\B (\zz,\xx,\yy)^{p}d\zz \leq\int_{0}^{\xx}{\zz}^{k_0}\B (\zz,\xx,\yy)^{p}d\zz \leq E_{k_0,p}{\xx}^{k_0+1-p}.
			\end{align*}
			Therefore, for all $\xx\in(0,\infty)$, we have
			\begin{equation*}
				\int_{0}^{\xx}\xi(\zz )\B (\zz,\xx,\yy)^{p}d\zz \leq E_{k_0,p}  {\xx}^{k_0+1-p}.
			\end{equation*}
			Similarily, for all $\yy\in(0,\infty)$, we have
			\begin{equation*}
				\int_{0}^{\yy}\xi(\zz )\B (\zz,\yy,\xx)^{p}d\zz \leq E_{k_0,p} {{\yy}^{k_0+1-p}}.
			\end{equation*}
			We next deduce from~\eqref{eq:fp}, \eqref{mass_truncated_prop}, \eqref{eq:differential_uniform_integrability}, \Cref{thm:smallsizebehavious}, and the above estimates that
			\begin{align*}
					\frac{d}{dt}\mm_{\xi}(\psi&(\ff_n(t))) \leq (\rho+C_1(T))^2 \mm_{\xi}(\psi(\ff_n(t)))\\
					&\quad + B_{p}E_{k_0,p}\int_{0}^{\infty} \int_{0}^{\infty} \bigg[ {\xx}^{k_0+1-p} + {\yy}^{k_0+1-p} \bigg] {\xx}^{\lambda_1} {\yy}^{\lambda_2} \ff_n(t,\xx) \ff_n(t,\yy) d\yy d\xx \\
					& \le (\rho+C_1(T))^2 \mm_{\xi}(\psi(\ff_n(t))) + B_{p}E_{k_0,p} \mm_{k_0+1-p+\lambda_1}(\ff_n(t)) \mm_{\lambda_2}(\ff_n(t))\\
					& \quad + B_{p}E_{k_0,p} \mm_{\lambda_1}(\ff_n(t)) \mm_{k_0+1-p+\lambda_2}(\ff_n(t)) .
			\end{align*}
			Since $\lambda_i\in [k_0,1]$ and $k_0+1-p+\lambda_i\in [k_0,k_0+1]$ for $i\in \{1,2\}$, we further use Young's inequality, along with~\eqref{mass_truncated_prop}, \eqref{tail_truncated_prop}, and \Cref{thm:smallsizebehavious}, to find
			\begin{align*}
				\frac{d}{dt}\mm_{\xi}(\psi(\ff_n(t))) & \leq (\rho+C_1(T))^2 \mm_{\xi}(\psi(\ff_n(t)))\\
					& \quad + 2 B_{p}E_{k_0,p} \left( \mm_{k_0}(\ff_n(t)) + \mm_{k_0+1}(\ff_n(t)) \right) \left( \mm_{k_0}(\ff_n(t)) + \mm_{1}(\ff_n(t)) \right) \\
					&\leq (\rho+C_1(T))^2 \mm_{\xi}(\psi(\ff_n(t))) + 2B_{p}E_{k_0,p} (C_1(T) + \mm_{k_0+1}(\ff^{\mbox{\rm{\mbox{in}}}})) (\rho+C_1(T)).
			\end{align*}
			Therefore, by integrating the above inequality with respect to time, we complete the proof of \Cref{lem:uniform_int}.
		\end{proof}

		\subsection{Time equicontinuity} 
		
		Finally, we  prove the time equicontinuity of the sequence $(\ff_n)_{n\ge1}$.
		
		\begin{lem}\label{lem:equicontinuity}  Let $T \in (0, T_*)$. There exists a positive constant $C_4(T)$ which is independent of $n \geq 1$ and  depends only on $\kk$, $b$, $\ff^{\mbox{\rm{\mbox{in}}}}$, and $T$, such that the following inequality holds for all $ 0\le t_1 < t_2 \le T$,
			\begin{equation}
				\int_{0}^{\infty}{\xx}^{k_0} \lvert \ff_n (t_2,\xx)-\ff_n (t_1,\xx)\rvert d\xx \leq C_{4}(T)(t_2-t_1).
			\end{equation}
		\end{lem}
		
		\begin{proof}  For $t\in [0,T]$, we infer from \Cref{lem:m}, \Cref{thm:smallsizebehavious}, \eqref{eq:twf}, \eqref{mass_truncated_prop}, \eqref{tail_truncated_prop}, and~\eqref{eq:K_uniform_bound} that
			\begin{align*}
					\int_{0}^{\infty}{\xx}^{k_0} \lvert \partial_{t}&\ff_n(t,\xx) \rvert d\xx \\
					 &\leq\frac{1}{2}\int_{0}^{\infty} \int_{0}^{\infty}|\Upsilon_{W_{k_0}}(\xx,\yy)|\kk(\xx,\yy)\ff_n(t,\xx)\ff_n(t,\yy)d\yy d\xx \\
					&\leq \frac{E_{k_0}}{2}\int_{0}^{\infty} \int_{0}^{\infty}({\xx}^{k_0}+{\yy}^{k_0})\kk(\xx,\yy)\ff_n(t,\xx)\ff_n(t,\yy)d\yy d\xx \\
					&\leq E_{k_0}\int_{0}^{\infty} \int_{0}^{\infty}({\xx}^{k_0}+{\yy}^{k_0})(\xx+{\xx}^{k_0})(\yy+{\yy}^{k_0})\\
					& \hspace{5cm} \times \ff_n(t,\xx)\ff_n(t,\yy)d\yy d\xx \\
					&\leq 2E_{k_0}\int_{0}^{\infty} \int_{0}^{\infty}{\xx}^{k_0} (\xx+{\xx}^{k_0})(\yy+{\yy}^{k_0})\ff_n(t,\xx)\ff_n(t,\yy)d\yy d\xx \\
					&\leq 2E_{k_0}(\rho + C_1(T)) \bigg( \mm_{1+k_0}(\ff^{\mbox{\rm{\mbox{in}}}}) + \mm_{2k_0}(\ff_n(t)) \bigg).
			\end{align*}
			Since $k_0<2k_0<k_0+1$, by H\"older's inequality, we get, using once more \Cref{thm:smallsizebehavious} and~ \eqref{tail_truncated_prop},
			\begin{equation*}
				\mm_{2k_{0}}(\ff_n(t)) \le \mm_{k_0}(\ff_n(t))^{1-k_0} \mm_{1+k_0}(\ff_n(t))^{k_0} \le C_{1}(T)^{1-k_0} \mm_{1+k_0}(\ff^{\mbox{\rm{\mbox{in}}}})^{k_0}.
			\end{equation*}
			Consequently,
			\begin{equation*}
					\int_{0}^{\infty}{\xx}^{k_0} \lvert \partial_{t}\ff_n(t,\xx) \rvert d\xx \leq 2E_{k_0} (\rho + C_1(T)) \bigg( \mm_{1+k_0}(\ff^{\mbox{\rm{\mbox{in}}}}) + C_{1}(T)^{1-k_0} \mm_{1+k_0}(\ff^{\mbox{\rm{\mbox{in}}}})^{k_0} \bigg).
			\end{equation*}
			 Since 
			\begin{equation*}
				\int_0^\infty {\xx}^{k_0} | u_n(t_2,\xx) - u_n(t_1,\xx)| d\xx \le \int_{t_1}^{t_2} \int_{0}^{\infty}{\xx}^{k_0} \lvert \partial_{t}\ff_n(t,\xx) \rvert d\xx dt
			\end{equation*}
			for $0\le t_1 < t_2 \le T$, integrating the above inequality with respect to time over $(t_1,t_2)$ provides the result of \Cref{lem:equicontinuity}.
		\end{proof}
		
		The results obtained so far can now be utilized to establish \Cref{maintheorem}, which we now proceed to prove.
		
		\begin{proof}[Proof of \Cref{maintheorem} ]
			\textbf{Compactness:}  Recall that $\xi(\xx)=\min\{\xx,{\xx}^{k_0}\}$ for $\xx\geq 0$. Let $T\in(0,T_*)$ be fixed, and let us define
			\begin{equation*}
				\mathcal{E}(T):=\{\ff_n(t) : t\in [0,T], n\geq 1 \}.
			\end{equation*}
			Consider a measurable subset $A$ of $(0,\infty)$ with finite Lebesgue measure, and let $M>1$. Then, due to~\eqref{mass_truncated_prop}, \Cref{lem:uniform_int}, and the monotonicity of $s \mapsto \psi(s)/s$, we have, for any $t\in [0,T]$ and $n\geq 1$, 
			\begin{align*}
					\int_{A}\xi(\xx) \ff_n(t,\xx)d\xx  &\leq \int_{A \cap (0,M)}\xi(\xx) \ff_n(t,\xx)d\xx +\int_{M}^{\infty}\xi(\xx)\ff_n(t,\xx)d\xx \\
					& \leq \int_{A \cap (0,M)}\xi(\xx) \ff_n(t,\xx)\textbf{1}_{(0,M)}(\ff_n(t,\xx))d\xx \\
					& \quad +\int_{A \cap (0,M)}\xi(\xx) \ff_n(t,\xx)\textbf{1}_{[M,\infty)}(\ff_n(t,\xx))d\xx +\frac{\rho}{M^{1-k_0}}\\
					&\leq M^2 \lvert A \rvert +\frac{M}{\psi(M)}\int_{0}^{M}\xi(\xx)\psi(\ff_n(t,\xx))d\xx +\frac{\rho}{M^{1-k_0}}\\
					& \leq M^2\lvert A \rvert +\frac{M}{\psi(M)}C_3(T)+\frac{\rho}{M^{1-k_0}}.
			\end{align*}
			Therefore the modulus of uniform integrability 
			\begin{equation*}
				\eta\{\mathcal{E}(T);X_\xi\} := \lim_{\varepsilon \to 0} \sup\bigg\{ \int_A \xi(\xx) \g(\xx)d\xx  : \g\in \mathcal{E}(T), A\subset(0,\infty), |A|\leq \varepsilon\bigg\}
			\end{equation*}
			of $\mathcal{E}(T)$ in $X_{\xi}$ satisfies
			\begin{equation*}
				\eta\{\mathcal{E}(T);X_\xi\}\leq \frac{M}{\psi(M)}C_3(T)+\frac{\rho}{M^{1-k_0}}
			\end{equation*}
			for all $ M>1$. Recalling that $\psi$ satisfies~\eqref{eq:monotonocity_psi} and $k_0\in (0,1)$, we may let $M \to \infty$ in the previous inequality to obtain that
			\begin{equation}
				\eta\{\mathcal{E}(T);X_\xi\}=0.\label{eq:cw1}
			\end{equation}
			Also,  by~\eqref{mass_truncated_prop},
			\begin{equation*}
				\int_{M}^{\infty}\xi(\xx)\ff_n(t,\xx)d\xx \leq \frac{\rho}{M^{1-k_0}}
			\end{equation*}
			for $n\geq 1$, $t\in[0,T]$, and $M>1$, so that
			\begin{equation}
				\lim_{M \to \infty} \sup_{\g\in \mathcal{E}(T)} \Big\{\int_{M}^{\infty}\xi(\xx)\g(\xx)d\xx \Big \}=0. \label{eq:cw2}
			\end{equation}
			According to~\eqref{eq:cw1}, \eqref{eq:cw2}, and the Dunford-Pettis theorem \cite[Theorem~7.1.3]{bll2019}, $\mathcal{E}(T)$ is a relatively sequentially weakly compact subset of $X_{\xi}$. Applying a version of the Arzel\`a-Ascoli theorem \cite[Theorem~7.1.16]{bll2019} , we can deduce from \Cref{lem:equicontinuity} that $(\ff_n)_{n\geq1}$ is relatively sequentially compact in $\mathcal{C}([0,T],X_{\xi,w})$. Since $T$ is arbitrary in $(0,T_*)$, we can use a diagonal process to obtain a subsequence of $(\ff_n)_{n\geq 1}$ (not relabeled) and $\ff\in \mathcal{C}([0,T_*),X_{\xi,w})$ such that
			\begin{equation}
				\lim_{n \to \infty} \sup_{t\in[0,T]} \Big \lvert \int_{0}^{\infty}\xi(\xx) (\ff_n-\ff)(t,\xx)\tf(\xx)d\xx  \Big \rvert=0 \label{eq:convergence}
			\end{equation}
			 for all $\tf\in L^\infty((0,\infty))$ and $T\in (0,T_*)$. Thanks to~\eqref{eq:convergence}, we extend the validity of~\eqref{tail_truncated_prop} and \Cref{ssbrevisted} from $\ff_n$ to $\ff$ as shown in \cite{GL22021} and obtain that, for $T \in (0,T_*)$,
				\begin{equation}
				\int_{\xx}^{\infty}\yy^{k_0+1}\ff(t,\yy)d\yy\le \int_{\xx}^{\infty}\yy^{k_0+1}\ff^{\mbox{\rm{\mbox{in}}}}(\yy)d\yy, \qquad (t,\xx)\in [0,T]\times (0,\infty), \label{tail_limit_truncated_prop}
			\end{equation}
			and 
			\begin{equation}
				\mm_{\psi_0}(\ff(t)) \leq C_2(T), \qquad t\in [0,T]. \label{eq:small_size_revisited_limit_truncated}
			\end{equation}
			
			We now extend~\eqref{eq:convergence} to the weak topology of $X_{k_0} \cap X_{k_0+1}$, i.e.,
			\begin{equation}
				\lim_{n \to \infty} \sup_{t\in[0,T]} \Big \lvert \int_{0}^{\infty}({\xx}^{k_0}+{\xx}^{k_0+1})(\ff_n-\ff)(t,\xx)\tf(\xx)d\xx  \Big \rvert=0  \label{eq:X1mm+1convergence}
			\end{equation}
			for all $\tf  \in L^\infty((0,\infty))$ and $T\in (0,T_*)$. Indeed, let $T\in (0,T_*)$, $t\in [0,T]$, $\tf  \in L^\infty((0,\infty))$, and $M>1$. Then, 
			\begin{align*}
				&	\left \lvert \int_{0}^{\infty}({\xx}^{k_0}+{\xx}^{k_0+1})(\ff_n-\ff)(t,\xx)\tf(\xx)d\xx  \right \rvert \nonumber \\
				& \qquad \qquad \leq  \underbrace{2\lvert \lvert \tf \rvert  \rvert_{L^\infty}\int_{0}^{1/M}{\xx}^{k_0}(\ff_n+\ff)(t,\xx)d\xx}_{I_{1,n}(t)} \nonumber\\ 
				& \qquad\qquad\qquad + \underbrace{  {\left\lvert \int_{1/M}^{M}({\xx}^{k_0}+{\xx}^{k_0+1}) (\ff_n-\ff)(t,\xx) \tf(\xx)d\xx \right\rvert}}_{I_{2,n}(t)} \nonumber\\ 
				& \qquad\qquad\qquad + \underbrace{2\lvert \lvert \tf \rvert \rvert_{L^\infty}\int_{M}^{\infty}{\xx}^{k_0+1}(\ff_n+\ff)(t,\xx)d\xx}_{I_{3,n}(t)}.
			\end{align*}
			First, thanks to~\eqref{eq:small_size_revisited_truncated}, \eqref{eq:small_size_revisited_limit_truncated}, and the monotonicity of $\psi_0$,
			\begin{align}
				I_{1,n}(t)&:= 2\lvert \lvert \tf \rvert \rvert_{L^\infty}\int_{0}^{1/M}{\xx}^{k_0}(\ff_n+\ff)(t,\xx)d\xx  \nonumber \\
				&\leq \frac{2\lvert \lvert \tf \rvert \rvert_{L^\infty}}{M^{k_0}\psi_0(1/M)} \int_{0}^{1/M} \psi_{0}(\xx) (\ff_n+\ff)(t,\xx)d\xx \nonumber
				\\ 
				&\leq \frac{2\lvert \lvert \tf \rvert \rvert_{L^\infty}}{M^{k_0}\psi_0(1/M)} [\mm_{\psi_{0}}(\ff_n(t))+\mm_{\psi_{0}}(\ff (t))]\leq \frac{4\lvert \lvert \tf \rvert \rvert_{L^\infty}}{M^{k_0}\psi_0(1/M)} C_2(T). \label{eq:I1}
			\end{align}
			Next, by~\eqref{tail_truncated_prop} and~\eqref{tail_limit_truncated_prop},
			\begin{equation}
				I_{3,n}(t) :=  2\lvert \lvert \tf \rvert \rvert_{L^\infty} \int_{M}^{\infty} {\xx}^{k_0+1} (\ff_n+\ff)(t,\xx)d\xx \leq 4\lvert \lvert \tf \rvert \rvert_{L^\infty} \int_{M}^{\infty} {\xx}^{k_0+1}\ff^{\mbox{\rm{\mbox{in}}}}(\xx)d\xx. \label{eq:I3}
			\end{equation}
			Finally, observe that 
			\begin{equation*}
					\left\lvert \frac{{\xx}^{k_0}+{\xx}^{k_0+1}}{\xi(\xx)}\textbf{1}_{(1/M,M)}(\xx)\tf(\xx) \right \rvert \leq  2M \lvert  \lvert  \tf \rvert \rvert_{L^\infty},
				\end{equation*}
		which implies that 
			\begin{equation}
					\lim_{n\to\infty} \sup_{t\in [0,T]} I_{2,n}(t) = 0, \label{eq:I2}
				\end{equation}
				according to~\eqref{eq:convergence}. Gathering~\eqref{eq:I1}, \eqref{eq:I3}, and~\eqref{eq:I2},  we obtain that
			\begin{align*}
				 \limsup_{n \to \infty} \sup_{t\in[0,T]} \Big \lvert \int_{0}^{\infty}&({\xx}^{k_0}+{\xx}^{k_0+1})(\ff_n-\ff)(t,\xx)\tf(\xx)d\xx  \Big \rvert \\
				&\le 4\lvert \lvert \tf \rvert \rvert_{L^\infty} \Big[\frac{C_2(T)}{M^{k_0}\psi_0(1/M)}+ \int_{M}^{\infty} {\xx}^{k_0+1} \ff^{\mbox{\rm{\mbox{in}}}}(\xx)d\xx \Big ].
			\end{align*}
			As $\ff^{\mbox{\rm{\mbox{in}}}}$ satisfies~\eqref{eq:additionalinitiadata} and $\psi_0$ satisfies~\eqref{eq:monotonocity_psi_0}, the right-hand side of the above inequality tends to zero as $M \to \infty$, thereby concluding the proof of~\eqref{eq:X1mm+1convergence}.
			
			\medskip
			
			\textbf{Mass conservation:}  In particular, the convergence established in~\eqref{eq:X1mm+1convergence} implies that $\ff \in \mathcal{C}([0,T_* ),X_{1,w})$ and a straightforward consequence of~\eqref{mass_truncated_prop} is that
				\begin{equation}
					\mm_1(\ff (t))=\lim_{n \to \infty} \mm_1(\ff_n(t))=\lim_{n \to \infty} \mm_1(\ff_n^{\mbox{\rm{\mbox{in}}}})= \mm_1(\ff^{\mbox{\rm{\mbox{in}}}}(t))=\rho,\qquad t\in [0,T_*). \label{eq:limitmass}
			\end{equation}
			
			\medskip
			
			 \textbf{Regularity and integrability properties:} According to~\eqref{eq:X1mm+1convergence} and~\eqref{eq:limitmass}, $u$ belongs to $C([0,T_*),X_{k_0,w})$ and to $L^\infty((0,T_*),X_1)$ and thus satisfies~\eqref{eq:ws1}. In addition, it follows from~\eqref{eq:kernel}, \eqref{eq:K_uniform_bound}, and~\eqref{tail_limit_truncated_prop} that $\ff$ satisfies~\eqref{eq:ws2}.
			
			\medskip
			
			\textbf{Limit Equation:} Now, we show that the limit function $u$ satisfies the weak formulation~\eqref{eq:wf}. Let $t\in (0,T_*)$ and $\tf \in \mathscr{T}^{k_0}$. On the one hand, it follows from~\eqref{eq:tin} and~\eqref{eq:X1mm+1convergence}  that
			\begin{equation*}
				\lim_{n \to \infty}\int_{0}^{\infty}\tf(\xx)(\ff_n(t,\xx)-\ff_n^{\mbox{\rm{\mbox{in}}}}(\xx))d\xx = \int_{0}^{\infty}\tf(\xx)(\ff (t,\xx)-\ff^{\mbox{\rm{\mbox{in}}}}(\xx))d\xx .
			\end{equation*}
			On the other hand, we observe that
			\begin{align*}
				\lim_{n\to \infty} \frac{\Upsilon_{\tf}(\xx,\yy)\kk_n(\xx,\yy)}{({\xx}^{k_0}+{\xx}^{k_0+1})({\yy}^{k_0}+{\yy}^{k_0+1})}=\frac{\Upsilon_{\tf}(\xx,\yy)\kk(\xx,\yy)}{({\xx}^{k_0}+{\xx}^{k_0+1})({\yy}^{k_0}+{\yy}^{k_0+1})}, 
			\end{align*}
			for $(\xx,\yy)\in (0,\infty)^2$ and 
			\begin{align*}
				\frac{({\xx}^{k_0}+{\yy}^{k_0})(\xx+{\xx}^{k_0})(\yy+{\yy}^{k_0})}{({\xx}^{k_0}+{\xx}^{k_0+1})({\yy}^{k_0}+{\yy}^{k_0+1})} \le 8 \quad \text{for} \quad (\xx,\yy)\in (0,\infty)^2,
			\end{align*}
			which implies that, by \Cref{lem:m} and~\eqref{eq:K_uniform_bound},
			\begin{align*}
				\left\lvert \frac{\Upsilon_{\tf}(\xx,\yy) \kk_n(\xx,\yy)}{({\xx}^{k_0}+{\xx}^{k_0+1}) ({\yy}^{k_0}+{\yy}^{k_0+1})} \right\rvert &\leq 2E_{k_0}\lvert \lvert \tf \rvert \rvert_{C^{0,k_0}}\frac{({\xx}^{k_0}+{\yy}^{k_0})(x+{\xx}^{k_0})(\yy+{\yy}^{k_0})}{({\xx}^{k_0}+{\xx}^{k_0+1})({\yy}^{k_0}+{\yy}^{k_0+1})}\\
				&\leq 16E_{k_0}\lvert \lvert \tf \rvert \rvert_{C^{0,k_0}}.
			\end{align*}
			Since the convergence obtained in~\eqref{eq:X1mm+1convergence} implies that
			\begin{equation*}
				[(\tau,\xx,\yy) \mapsto \ff_n(\tau,\xx)\ff_n(\tau,\yy)] \rightharpoonup  [(\tau,\xx,\yy) \mapsto  {\ff(\tau,\xx)\ff(\tau,\yy)}]
			\end{equation*}
			in $L^1((0, t) \times (0, \infty)^2,({\xx}^{k_0}+{\xx}^{k_0+1})({\yy}^{k_0}+{\yy}^{k_0+1})d\yy d\xx d\tau)$, we infer from \cite[Proposition~2.61]{FL2007} that
			\begin{align*}
				\lim_{n\to \infty} \frac{1}{2} \int_{0}^{t} \int_{0}^{\infty} \int_{0}^{\infty} & \Upsilon_{\tf}(\xx,\yy) \kk_n(\xx,\yy) \ff_n(\tau,\xx) \ff_n(\tau,\yy) d\yy d\xx d\tau \\
				&=\frac{1}{2} \int_{0}^{t} \int_{0}^{\infty} \int_{0}^{\infty} \Upsilon_{\tf}(\xx,\yy) \kk(\xx,\yy) \ff (\tau,\xx) \ff (\tau,\yy) d\yy d\xx d\tau.
			\end{align*}
			
			 We have thus verified that $\ff$ satisfies all the conditions for being a weak solution to~\eqref{eq:main}-\eqref{eq:in} on $[0, T_*)$, and it is also mass-conserving on $[0,T_*)$, according to~\eqref{eq:limitmass}.
		\end{proof}

		\section{Uniqueness}\label{sec:u}
		
		\begin{proof}[Proof of \Cref{thm:uniqueness} ]
			
			Let $\ff_1$ and $\ff_2$ be two weak solutions to~\eqref{eq:main}-\eqref{eq:in} in the sense of \Cref{defn:weaksolution} on $[0,T_1)$ and $[0,T_2)$, respectively, both satisfying \eqref{eq:moment1+k0+lambda2} for each $t \in  (0, \min\{ T_1 , T_2 \} )$. We set $e_1:= \ff_1 - \ff_2$, $e_2:=\ff_1+\ff_2$, $\Sigma:= sign(\ff_1 -\ff_2)$, $w(\xx) := \max\{{\xx}^{k_0} , {\xx}^{1+k_0}\}$, $\xx>0$,
			and we infer from~\eqref{eq:wf} that, for $t \in  (0, \min\{ T_1 , T_2 \} )$,
			\begin{equation}\label{eq:uniqueness}
				\frac{d}{dt}\int_{0}^{\infty}w(\xx)|e_1(t,\xx)|\leq \frac{1}{2}\int_{0}^{\infty}\int_{0}^{\infty}Q(t,\xx,\yy)e_2(t,\xx) \lvert e_1(t,\yy)\rvert d\yy d\xx ,
			\end{equation}
			where
			\begin{align*}
				Q(t,\xx,\yy): &=\kk(\xx,\yy)\Upsilon_{w\Sigma(t)}(\xx,\yy)\Sigma(t,\yy)\\
				&\leq \kk(\xx,\yy)\bigg[\int_{0}^{\xx+\yy}w(\zz)\bb (\zz,\xx,\yy)d\zz +w(\xx)-w(\yy)\bigg].
			\end{align*}
			We now estimate $Q(t,\xx,\yy)$ with the help of~\eqref{eq:masstransfer}, \eqref{eq:nonintegrable}, and~\eqref{eq:kernel} and split the analysis according to the range of $(\xx,\yy)$.
		
				\textit{Case 1.} If $(\xx,\yy)\in(0,1)^2$, then
				\begin{align*}
						Q(t,\xx,\yy)&\leq \kk(\xx,\yy)\bigg[\int_{0}^{\xx+\yy}{\zz}^{k_0} \bb (\zz,\xx,\yy)d\zz +{\xx}^{k_0}-{\yy}^{k_0}\bigg]\\
						&  \le 2 (\xx\yy)^{k_0} \big[ (E_{k_0,1}+1) \xx^{k_0} + (E_{k_0,1}-1) \yy^{k_0} \big] \\
						&\leq 4E_{k_0,1}{\xx}^{k_0} w(\yy),
				\end{align*}
				 recalling that $E_{k_0,1}\ge 1$.
				
				\textit{Case 2.} If $(\xx,\yy)\in(0,1)\times (1,\infty)$, then 
				\begin{align*}
						Q(t,\xx,\yy)&\leq \kk(\xx,\yy)\bigg[\int_{0}^{\xx}{\zz}^{k_0} \bb_* (\zz,\xx,\yy)d\zz + {\xx}^{k_0} \bigg] \\
						& \hspace{-0.5cm} + \kk(\xx,\yy) \bigg[ \int_{0}^1 \zz^{k_0} \bb_*(\zz,\yy,\xx) d\zz + \int_{1}^{\yy}{\zz}^{1+k_0} \bb_* (\zz,\xx,\yy)d\zz -{\yy}^{k_0+1}\bigg]\\
						& \hspace{-1cm}  \le \kk(\xx,\yy) \bigg[ (E_{k_0,1}+1) \xx^{k_0} + E_{k_0,1} \yy^{k_0} + \yy^{k_0} \int_1^{\yy} \zz \beta_*(\zz,\yy,\xx) d\zz - \yy^{k_0+1} \bigg] \\
						& \hspace{-1cm} \le 3 E_{k_0,1} \xx^{k_0} \left( \yy^{\lambda_1} + \yy^{\lambda_2} \right) \yy^{k_0} \\
						& \hspace{-1cm} \leq 6E_{k_0,1}{\xx}^{k_0} w(\yy).
				\end{align*}

			\textit{Case 3.} If $(\xx,\yy)\in(1,\infty)\times (0,1)$, then 
				\begin{align*}
						Q(t,\xx,\yy)&\leq \kk(\xx,\yy)\bigg[\int_{0}^{\yy}{\zz}^{k_0} \bb_* (\zz,\yy,\xx)d\zz - {\yy}^{k_0} \bigg] \\
						& \hspace{-0.5cm} + \kk(\xx,\yy) \bigg[ \int_{0}^1 \zz^{k_0} \bb_*(\zz,\xx,\yy) d\zz + \int_{1}^{\xx}{\zz}^{1+k_0} \bb_* (\zz,\yy,\xx)d\zz + {\xx}^{k_0+1}\bigg]\\
						& \hspace{-1cm}  \le \kk(\xx,\yy) \bigg[ (E_{k_0,1}-1) \yy^{k_0} + E_{k_0,1} \xx^{k_0} + \xx^{k_0} \int_1^{\xx} \zz \beta_*(\zz,\xx,\yy) d\zz + \xx^{k_0+1} \bigg] \\
						& \hspace{-1cm} \le \yy^{k_0} \left( \xx^{\lambda_1} + \xx^{\lambda_2} \right) \big[ E_{k_0,1}-1 + E_{k_0,1} + 2 \big] \xx^{k_0+1} \\
						& \hspace{-1cm} \le 6 E_{k_0,1}{\xx}^{1+k_0+{\lambda_2}} w(\yy).
				\end{align*}
				
				\textit{Case 4.} If $(\xx,\yy)\in(1,\infty)^2$, then 
				\begin{align*}
						Q(t,\xx,\yy)&\leq \kk(\xx,\yy) \bigg[ E_{k_0,1} \xx^{k_0} + \xx^{k_0} \int_{1}^{\xx} {\zz} \bb_* (\zz,\xx,\yy)d\zz + {\xx}^{k_0+1} \bigg] \\
						& \hspace{-0.5cm} + \kk(\xx,\yy) \bigg[ E_{k_0,1} \yy^{k_0} + \yy^{k_0} \int_{1}^{\yy} {\zz} \bb_* (\zz,\yy,\xx)d\zz - {\yy}^{k_0+1} \bigg] \\
						& \hspace{-1cm} \le 2 (\xx\yy)^{\lambda_2} \big[ (E_{k_0,1}+2) \xx^{k_0+1} + E_{k_0,1} \yy^{k_0} \big] \\
						& \hspace{-1cm} \le 8 E_{k_0,1} (\xx\yy)^{\lambda_2}  \xx^{k_0+1} \yy^{k_0} \\
						& \hspace{-1cm} \leq 8 E_{k_0,1}{\xx}^{1+k_0+{\lambda_2}} w(\yy).
				\end{align*}
		
			Using these bounds in~\eqref{eq:uniqueness}, we get the following differential inequality
			\begin{equation*}
				\begin{split}
					\frac{d}{dt}\int_{0}^{\infty}w(\xx)&|e_1(t,\xx)|d\xx\\
					&\leq 12 E_{k_0,1} \int_{0}^{\infty} \int_{0}^{\infty} ({\xx}^{k_0}+{\xx}^{1+k_0+{\lambda_2}})w(\yy)e_2(t,\xx)|e_1(t,\yy)|d\yy d\xx \\
					&=12E_{k_0,1}[\mm_{k_0}(e_2(t))+\mm_{1+k_0+{\lambda_2}}(e_2(t))]\int_{0}^{\infty}w(\yy)|e_1(t,\yy)|d\yy. 
				\end{split}
			\end{equation*}
			
			The proof can be completed by applying Gronwall's lemma, since both $\mm_{k_0}(e_2)$ and $\mm_{1+k_0+{\lambda_2}}(e_2)$ are in $L^1(0,t)$, and $\ff_1(0) = \ff_2(0) = \ff^{\mbox{\rm{\mbox{in}}}}$.
		\end{proof}
		
\section{Non-existence}\label{sec:ne}

		 In this section, $\beta$ and $\kk$ are given by~\eqref{eq:masstransfer}, \eqref{eq:powerlaw}, and~\eqref{eq:kernel} with the parameters $(\lambda_1,\lambda_2,\nu)$ satisfying
		\begin{equation}
			\begin{split}
				\nu\in (-2,-1], \quad \lambda_1 \le \lambda_2 \le 1, \quad \lambda_1 < |\nu|-1, \quad \lambda = \lambda_1+\lambda_2 < 1.
			\end{split}\label{eq:ane}
		\end{equation}
		We fix $k_0\in (|\nu|-1,1)$ and consider a mass-conserving weak solution $\ff$ to~\eqref{eq:main}-\eqref{eq:in} in the sense of~\Cref{defn:weaksolution} on $[0,T_1)$ for some $T_1>0$ with initial condition $\ff^{\mbox{\rm{\mbox{in}}}}\in X_{k_0,+}\cap X_{k_0+1}$ and $\rho = \mm_1(\ff^{\mbox{\rm{\mbox{in}}}})>0$. Then
		\begin{equation}
			\mm_1(\ff (t)) = \rho =\mm_1(\ff^{\mbox{\rm{\mbox{in}}}})\,, \qquad t\in [0,T_1)\,, \label{eq:maasconservation_non}
		\end{equation}
		and, by~\Cref{lem:mtail}, 
		\begin{equation}
			\mm_{k_0+1}(\ff (t))\leq  \mm_{k_0+1}(\ff^{\mbox{\rm{\mbox{in}}}})\,, \qquad t\in [0,T_1)\,. \label{eq:moment_non}
		\end{equation}
		As already mentioned, the proof of \Cref{thm:nonexistence} is adapted from \cite{CadC92, vanD87c} and the first step is to show that all moments of $\ff$ of negative order, as well as sublinear moments, have to be finite.
		
		\begin{lem}\label{lem:negativemoments}
			For any $T\in (0,T_1)$  and $k\in (-\infty,1)$,
			\begin{equation*}
				\sup_{t\in [0,T]} \mm_k(\ff (t)) < \infty\,.
			\end{equation*}
			Particularly, $\ff^{\mbox{\rm{\mbox{in}}}}\in X_k$ for all $k\in (-\infty,1+k_0]$.
		\end{lem}
		
		\begin{proof}
			The proof is the same as that of \cite[Lemma~7.1]{GL22021} to which we refer.
		\end{proof}
		
		Now, we are going to prove \Cref{thm:nonexistence}.

		\begin{proof}[Proof of \Cref{thm:nonexistence} ]
			Consider $k\in (|\nu|-1,k_0)$, $\theta\in (0,1)$, and  define the function
			\begin{equation*}			
				\tf_{k,\theta}(\xx) = (\xx+\theta)^k - \theta^k ~\text{for} ~ \xx\in (0,1/\theta) ~\text{ and} ~ \tf_{k,\theta}(\xx)=\tf_{k,\theta}(1/\theta) ~ \text{otherwise}.	
			\end{equation*}
			 Clearly, $\tf_{k,\theta}\in L^\infty((0,\infty))$ and straightforward computations show that
			\begin{equation*}
				|\tf_{k,\theta}(\xx) - \tf_{k,\theta}(\yy)| \le \theta^{k-k_0} |\xx-\yy|^{k_0}, \qquad (\xx,\yy)\in (0,\infty)^2. 
			\end{equation*}
			 Consequently, $\tf_{k,\theta}\in \mathscr{T}^{k_0}$ and it follows from~\eqref{eq:wf} that, for $t\in (0,T_1)$,
			\begin{equation}
				\begin{split}
					& \int_0^\infty \tf_{k,\theta}(\xx) \ff(t,\xx)d\xx = \int_0^\infty \tf_{k,\theta}(\xx) \ff^{\mbox{\rm{\mbox{in}}}}(\xx)d\xx \\
					& \qquad\qquad + \frac{1}{2} \int_0^t \int_0^\infty \int_0^\infty \Upsilon_{\tf_{k,\theta}}(\xx,\yy) \kk(\xx,\yy) \ff(\tau,\xx) \ff(\tau,\yy)d\yy d\xx d\tau.
				\end{split} \label{eq:ne1}
			\end{equation}
			On the one hand, since 
			\begin{equation*}
				\tf_{k,\theta}(\xx) \le \xx^k \;\;\text{ and }\;\; \lim_{\theta\to 0} \tf_{k,\theta}(\xx) = \xx^k, \quad \xx\in (0,\infty),
			\end{equation*}
			we infer from \Cref{lem:negativemoments} and the Lebesgue dominated convergence theorem that 
			\begin{equation}
				\begin{split}
					\lim_{\theta\to 0} \int_0^\infty \tf_{k,\theta}(\xx) \ff(t,\xx)d\xx & = \mm_k(\ff(t)), \\
					\lim_{\theta\to 0} \int_0^\infty \tf_{k,\theta}(\xx) \ff^{\mbox{\rm{\mbox{in}}}}(\xx)d\xx & = \mm_k(\ff^{\mbox{\rm{\mbox{in}}}}).
				\end{split} \label{eq:ne2}
			\end{equation}
			On the other hand, since $\nu+k+1>0$,
			\begin{align*}
				\Upsilon_{\tf_{k,\theta}}(\xx,\yy) & = \frac{\nu+2}{\xx^{\nu+1}} \int_{0}^{\xx} \tf_{k,\theta}(\zz) \zz^\nu d\zz + \frac{\nu+2}{\yy^{\nu+1}} \int_{0}^{\yy} \tf_{k,\theta}(\zz) \zz^\nu d\zz \\
				& \qquad - \tf_{k,\theta}(\xx) - \tf_{k,\theta}(\yy)
			\end{align*}
			satisfies
			\begin{equation*}
				\big| \Upsilon_{\tf_{k,\theta}}(\xx,\yy) \big| \le \left( 1 + \frac{\nu+2}{k+\nu+1} \right) \big( \xx^k + \yy^k \big), \qquad (\xx,\yy)\in (0,\infty)^2, 
			\end{equation*}
			so that
			\begin{align*}
				\big| \Upsilon_{\tf_{k,\theta}}(\xx,\yy) \big| \kk(\xx,\yy) & \le \left( 1 + \frac{\nu+2}{k+\nu+1} \right) \big( \xx^{k+\lambda_1} \yy^{\lambda_2} + \xx^{k+\lambda_2} \yy^{\lambda_1} \big)\\
				& \qquad + \left( 1 + \frac{\nu+2}{k+\nu+1} \right) \big( \xx^{\lambda_1} \yy^{k+\lambda_2} + \xx^{\lambda_2} \yy^{k+\lambda_1} \big)
			\end{align*}
			for $(\xx,\yy)\in (0,\infty)^2$. Since $\ff\in L^\infty((0,t),X_m)$ for $m\in\{\lambda_1,\lambda_2,k+\lambda_1,k+\lambda_2\}$ by~\eqref{eq:ane}, \eqref{eq:moment_non}, and~\Cref{lem:negativemoments}, and
			\begin{equation*}
				\lim_{\theta\to 0} \Upsilon_{\tf_{k,\theta}}(\xx,\yy) = \frac{1-k}{k+\nu+1} \big( \xx^k + \yy^k \big), \qquad (\xx,\yy)\in (0,\infty)^2,
			\end{equation*}
			we use once more the Lebesgue dominated convergence theorem to conclude that
			\begin{equation}
				\begin{split}
					& \lim_{\theta\to 0} \int_0^t \int_0^\infty \int_0^\infty \Upsilon_{\tf_{k,\theta}}(\xx,\yy) \kk(\xx,\yy) \ff(\tau,\xx) \ff(\tau,\yy)d\yy d\xx d\tau \\
					& \qquad  = \frac{1-k}{k+\nu+1} \int_0^t \int_0^\infty \int_0^\infty \big( \xx^k + \yy^k \big) \kk(\xx,\yy) \ff(\tau,\xx) \ff(\tau,\yy)d\yy d\xx d\tau.
				\end{split}\label{eq:ne3}
			\end{equation}
			Collecting~\eqref{eq:ne1}, \eqref{eq:ne2}, and~\eqref{eq:ne3}, we have shown that,  for $t\in (0,T_1)$,
			\begin{equation*}
				\begin{split}
					&  \mm_k(\ff(t)) = \mm_k(\ff^{\mbox{\rm{\mbox{in}}}}) \\
					& \qquad\qquad + \frac{1-k}{2(k+\nu+1)} \int_0^t \int_0^\infty \int_0^\infty \big( \xx^k + \yy^k \big) \kk(\xx,\yy) \ff(\tau,\xx) \ff(\tau,\yy)d\yy d\xx d\tau.
				\end{split} 
			\end{equation*}
			Hence, after using the symmetry of $\kk$ and~\eqref{eq:kernel},
			\begin{equation}
				\mm_k(\ff (t)) \ge \mm_k(\ff^{\mbox{\rm{\mbox{in}}}}) + \frac{1-k}{k+\nu+1} \int_0^t  \mm_{k+{\lambda_2}}(\ff (\tau)) \mm_{\lambda_1}(\ff (\tau))\ d\tau\,, \qquad t\in [0,T_1)\,. \label{eq: momentequation}
			\end{equation}
			Since ${\lambda_1} < |\nu|-1 < k < 1$, we can use H\"older's inequality and~\eqref{eq:maasconservation_non} to obtain, for $\tau\in [0,t)$,
			\begin{align*}
				\mm_k(\ff (\tau)) &\le \mm_{1}(\ff (\tau))^{(k-{\lambda_1})/(1-{\lambda_1})} \mm_{\lambda_1}(\ff (\tau))^{(1-k)/(1-{\lambda_1})}\\
				& \le  \rho^{(k-{\lambda_1})/(1-{\lambda_1})} \mm_{\lambda_1}(\ff (\tau))^{(1-k)/(1-{\lambda_1})}. 
			\end{align*} 
			Next, either $k<1<k+{\lambda_2}$ and we have
			\begin{equation*}
				\rho=\mm_1(\ff (\tau)) \le \mm_k(\ff (\tau))^{(k+{\lambda_2}-1)/{\lambda_2}} \mm_{k+{\lambda_2}}(\ff (\tau))^{(1-k)/{\lambda_2}}.
			\end{equation*}
			Or  $k+{\lambda_2} < 1 < 1+k_0$ and we have
			\begin{align*}
				\rho=\mm_1(\ff (\tau)) &\le \mm_{1+{k_0}}(\ff (\tau))^{(1-k-{\lambda_2})/(1+k_0-k-\lambda_2)} \mm_{k+{\lambda_2}}(\ff (\tau))^{k_0/(1+k_0-k-\lambda_2)}\\
				& \le \mm_{1+k_0}(\ff^{\mbox{\rm{\mbox{in}}}})^{(1-k-{\lambda_2})/(1+k_0-k-\lambda_2)} \mm_{k+{\lambda_2}}(\ff (\tau))^{k_0/(1+k_0-k-\lambda_2)}.
			\end{align*}
			Consequently,
			\begin{equation}
				\mm_{k+{\lambda_2}}(\ff (\tau)) \mm_{\lambda_1}(\ff (\tau)) \ge  \ell_2(k)  \mm_k(\ff (\tau))^{(1-k- {\ell_1(k))}/(1-k)}\,, \qquad \tau\in [0,t) \label{eq: k+beta alpha}
			\end{equation}
			with $ {\ell_1(k)} := \lambda_1-k+(k+\lambda_2-1)_+<0$ and
			\begin{equation*}
				 {\ell_2(k)} := \rho^{(\lambda_1-k)/(1-k)}\min\{\rho^{\lambda_2/(1-k)}, \rho^{(1+k_0-k-{\lambda_2})/k_0} \mm_{1+k_0}(\ff^{\mbox{\rm{\mbox{in}}}})^{(k+{\lambda_2}-1)/k_0}\} >0\,.
			\end{equation*}
			
			Combining~\eqref{eq: momentequation} and~\eqref{eq: k+beta alpha}, we derive an integral inequality for $\mm_k(\ff)$ expressed as
			\begin{equation}
				\mm_k(\ff(t)) \ge {\boldsymbol\eta}_k(t) := \mm_k(\ff^{\mbox{\rm{\mbox{in}}}}) + \frac{(1-k) \ell_2(k)}{k+\nu+1} \int_0^t \mm_k(\ff (\tau))^{(1-k- \ell_1(k))/(1-k)}\ d\tau \label{eq: MnNm}
			\end{equation}
			which holds for all $t \in [0,T_1)$. From~\eqref{eq: MnNm}, we infer that ${\boldsymbol\eta}_k$ satisfies the following differential inequality
			\begin{equation*}
				\frac{d{\boldsymbol\eta}_k}{dt}(t) \ge \frac{(1-k) {\ell_2(k)}}{k+\nu+1}  {\boldsymbol\eta}_k(t)^{ {(1-k- {\ell_1(k)})}/(1-k)}\,, \qquad t\in [0,T_1)\,.
			\end{equation*}
			By integrating the above differential inequality over the interval $[0,t]$, we derive that, for $t\in [0,T_1)$,
			\begin{equation*}
				{\boldsymbol\eta}_k(t)^{ {\ell_1(k)}/(1-k)} \le {\boldsymbol\eta}_k(0)^{ {\ell_1(k)}/(1-k)} + \frac{ {\ell_1(k)}  {\ell_2(k)}}{k+\nu+1} t,
			\end{equation*}
			 recalling that $ {\ell_1(k)}<0$. Using the above inequality and~\eqref{eq: MnNm}, we obtain
			\begin{equation*}
				\mm_k(\ff (t))^{ {\ell_1(k)}/(1-k)} \le {\boldsymbol\eta}_k(t)^{ {\ell_1(k)}/(1-k)} \le \mm_k(\ff^{\mbox{\rm{\mbox{in}}}})^{ {\ell_1(k)}/(1-k)} + \frac{ {\ell_1(k)}  {\ell_2(k)}}{k+\nu+1} t\,,
			\end{equation*}
			for $t\in [0,T_1)$. Since $\mm_k(\ff (t))^{ {\ell_1(k)}/(1-k)}\ge 0$, we have
			\begin{equation*}
				t \le \frac{k+\nu+1}{| {\ell_1(k)}|  {\ell_2(k)}} \mm_k(\ff^{\mbox{\rm{\mbox{in}}}})^{ {\ell_1(k)}/(1-k)}\,,
			\end{equation*}
			for $t\in [0,T_1)$. If we let $t\to T_1$ in the above inequality, then we conclude that
			\begin{equation}
				T_1 \le \frac{k+\nu+1}{| {\ell_1(k)}|  {\ell_2(k)}}  \mm_k(\ff^{\mbox{\rm{\mbox{in}}}})^{ {\ell_1(k)}/(1-k)}\,, \label{eq:ne4}
			\end{equation}
			 and this inequality is valid for any $k\in (|\nu|-1,1)$. Now, we note that
			\begin{equation*}
				\mm_k(\ff^{\mbox{\rm{\mbox{in}}}})^{k_0/(1+k_0-k)} \mm_{1+k_0}(\ff^{\mbox{\rm{\mbox{in}}}})^{(1-k)/(1+k_0-k)} \ge \mm_1(\ff^{\mbox{\rm{\mbox{in}}}}) = \rho,
			\end{equation*}
			so that
			\begin{equation*}
				\mm_k(\ff^{\mbox{\rm{\mbox{in}}}})^{ {\ell_1(k)}/(1-k)} \le \rho^{ {\ell_1(k)}(1+k_0-k)/k_0(1-k)}  \mm_{1+k_0}(\ff^{\mbox{\rm{\mbox{in}}}})^{| {\ell_1(k)}|/k_0}.
			\end{equation*}
			Consequently, since 
			\begin{align*}
				& \lim_{k\to -\nu-1}  {\ell_1(k)} = \lambda_1+ \nu+ 1 +(\lambda_2-\nu-2)_+<0\,, \\
				& \lim_{k\to -\nu-1}  { {\ell_2(k)}} = \rho^{(\lambda_1+\nu+1)/(\nu+2)}\min\{\rho^{\lambda_2/(\nu+2)}, \rho^{(\nu+2+k_0-{\lambda_2})/k_0} \mm_{1+k_0}(\ff^{\mbox{\rm{\mbox{in}}}})^{({\lambda_2}-\nu-2)/k_0}\} >0\,,
			\end{align*} 
			we may let $k\to -\nu-1$ in~\eqref{eq:ne4} and obtain
			\begin{equation*}
				T_1 \le \liminf_{k\to -\nu-1} \left\{ \frac{k+\nu+1}{| {\ell_1(k)}|  {\ell_2(k)}} \rho^{ {\ell_1(k)}(1+k_0-k)/k_0(1-k)}  \mm_{1+k_0}(\ff^{\mbox{\rm{\mbox{in}}}})^{| {\ell_1(k)}|/k_0} \right\} = 0,
			\end{equation*}
		thereby completing the proof of \Cref{thm:nonexistence}.
		\end{proof}
		

		
			\textbf{Acknowledgments.} 
			This research received financial support from the Indo-French Centre for Applied Mathematics (MA/IFCAM/19/58) as part of the project ``Collision-induced breakage and coagulation: dynamics and numerics." The authors would also like to acknowledge Council of Scientific \& Industrial Research (CSIR), India for providing a PhD fellowship to RGJ through Grant 09/143(0996)/2019-EMR-I. Part of this work was done while PhL enjoyed the hospitality of the Department of Mathematics, Indian Institute of Technology Roorkee.
		
			\bibliography{Refs.bib}
		\bibliographystyle{abbrv}

	\end{document}